\documentclass[11pt,a4paper]{article}
\usepackage[T1]{fontenc}
\usepackage[utf8]{inputenc}
\usepackage{lmodern}
\usepackage[english]{babel}
\usepackage[margin=1in]{geometry}
\usepackage{amsmath}

\usepackage{graphicx}        

\usepackage{amsfonts,amssymb}
\usepackage[matrix,arrow,curve]{xy}
\usepackage[numbers]{natbib}
\usepackage[hidelinks]{hyperref}
\usepackage{amsthm}

\newtheorem{theorem}{Theorem}[section] 
\newtheorem{definition}[theorem]{Definition} 
\newtheorem{lemma}[theorem]{Lemma}
\newtheorem{remark}[theorem]{Remark}
\newtheorem{proposition}[theorem]{Proposition}
\newtheorem{example}[theorem]{Example}
\newtheorem{corollary}[theorem]{Corollary}
\begin{document}

\title{Bousfield--Kan Completions of Subcontractible Presentations}
\author{Andrey Mikhovich\thanks{Moscow Center for Fundamental and Applied
Mathematics, Lomonosov Moscow State University, GSP-1, Leninskie Gory,
Moscow 119991, Russian Federation. Email: \texttt{amikhovich@gmail.com}.}}
\date{}
\maketitle

\begin{center}
\textit{Dedicated to Professor Alexander S. Mishchenko on the occasion of his
85th birthday, with deepest respect and admiration.}
\end{center}




\begin{abstract}
We study Bousfield--Kan completions through the interaction of free
simplicial resolutions, their filtration spectral sequences, and a
noncommutative arithmetic square.  For every free discrete simplicial
group of finite type, we express its integral pronilpotent completion
as the homotopy pullback of its rational prounipotent completion and
the product of its pro-$p$ completions over an explicit adelic
simplicial group.  The adelic entry is formed by taking restricted
products at finite nilpotent stages and then their inverse limit;
no nilpotency assumption on the original group of components is required.
Finite subpresentations of contractible presentations provide an
explicit application of this construction.  Independence of the
specified relators makes the positive-degree terms of the rational
and mod-$p$ filtration spectral sequences vanish, with convergence
verified on the quotient towers.  Continuous comparison of free
simplicial resolutions then realizes, in characteristic zero, the
equivalence with a constant free prounipotent group by morphisms and
homotopies in that category.  For the corresponding presentation
complex $K$ we obtain $R_\infty K\simeq K(F_R(Z),1)$ for
$R=\mathbb Q,\mathbb F_p,\mathbb Z$, where $Z$ indexes a complementary
basis and $F_R(Z)$ denotes, respectively, the rational points of a
free prounipotent group, a free pro-$p$ group, or a free pronilpotent
group.  Compatible contractions at the nilpotent stages identify all
four entries of the arithmetic square in this case.
\end{abstract}



\section{Introduction}\label{s0}
We study rational, pro-$p$, and integral Bousfield--Kan completions of
finite subpresentations of contractible presentations. The proofs use
free simplicial resolutions, their lower-central and Zassenhaus
filtrations, and a noncommutative arithmetic square. The square is
constructed for every free discrete simplicial group of finite type.

A presentation $(X\mid Y)$ of a discrete group $G$ determines an exact sequence
\begin{equation}\label{eq1}
1 \rightarrow R \rightarrow \Phi \xrightarrow{\pi} G \rightarrow 1
\end{equation}
in which $\Phi=\Phi(X)$ is the free discrete group with a set $X$ of generators, and $R$
is a normal subgroup in $\Phi$ normally generated by a set $Y\subset R$ of defining relations. 
We use the notation $(X\mid Y)$ for a presentation with generating set $X$
and chosen set of relators $Y$.  The normal closure of $Y$ in $\Phi(X)$ will
be denoted by $R$.  Thus \eqref{eq1} is the extension determined by the
presentation, rather than the presentation itself.
By a subpresentation $(X_1\mid Y_1)$ we mean a choice of subsets
$X_1\subseteq X$ and $Y_1\subseteq Y$ such that the selected relators
are words in $\Phi(X_1)$.  A presentation is called \emph{aspherical}
when its associated two-dimensional CW-complex $K(X\mid Y)$ has
$\pi_2K(X\mid Y)=0$.  The construction is recalled in
Section~\ref{r1.2}; see also \cite{Hut}.

A presentation is called \emph{contractible} when its presentation complex is
contractible.  A \emph{subcontractible presentation} is a finite
subpresentation of a contractible presentation.

J.~H.~C. Whitehead's asphericity conjecture asserts that every connected
subcomplex of an aspherical two-dimensional CW-complex is aspherical; see
\cite{Ro} and the references therein.  Howie showed that it is enough to
consider subcomplexes of contractible two-complexes \cite{Ho83}.

The asphericity of the completions has a short homological proof.
Let $K=K(X_1\mid Y_1)$ be subcontractible. The cellular boundary of
the ambient contractible presentation is an isomorphism over $\mathbb Z$;
the selected relators therefore give a split injection
$\mathbb Z^{Y_1}\hookrightarrow\mathbb Z^{X_1}$
(Lemma~\ref{lem:relator-independence}). Thus
$H_2(K;\mathbb Z)=0$ and $H_1(K;\mathbb Z)\cong\mathbb Z^r$,
where $r=|X_1|-|Y_1|$. Choose based loops representing a basis of
$H_1(K;\mathbb Z)$. They define a map
\[
 f:\bigvee_{i=1}^{r}S^1\longrightarrow K
\]
which is an integral homology equivalence: it is an isomorphism on
$H_1$, both spaces are connected, and their higher reduced homology
groups vanish. Bousfield--Kan completion carries an $R$-homology
equivalence to a weak equivalence \cite{BK}, so
\[
 R_\infty K\simeq R_\infty\!\left(\bigvee_{i=1}^{r}S^1\right)
 \simeq K(F_R(r),1),\qquad R=\mathbb Q,\mathbb F_p,\mathbb Z,
\]
where $F_R(r)$ is respectively the rational prounipotent, pro-$p$,
or integral pronilpotent completion of the free group of rank $r$
(with rational points understood in the first case).

This gives a ``black-box'' proof of the asphericity statement.
Our aim is to describe the mechanisms behind it: how the chosen
relators yield contractions in the completed simplicial models,
how the filtration spectral sequences detect the vanishing, and
how the rational and pro-$p$ models assemble in the arithmetic square.

\subsection{Presentations and their second homotopy groups}\label{r1.2}
Let $K=K(X\mid Y)$ be the presentation complex of
$G=\Phi(X)/R$, with relators $r_y$, $y\in Y$.
The cellular chains of its universal cover, together with
$\pi_2K\cong H_2(\widetilde K;\mathbb Z)$, give an exact sequence
of right $\mathbb ZG$-modules \cite[Proposition~9]{BH1}:
\begin{equation}\label{1.1}
0\longrightarrow\pi_2K\longrightarrow C_2(X\mid Y)
\xrightarrow{d_2}C_1(X\mid Y)\xrightarrow{d_1}\mathbb ZG
\xrightarrow{\epsilon}\mathbb Z\longrightarrow0,
\end{equation}
where $C_1=\bigoplus_{x\in X}e_x^1\mathbb ZG$ and
$C_2=\bigoplus_{y\in Y}e_y^2\mathbb ZG$.
With $\pi:\mathbb Z\Phi\to\mathbb ZG$ and right Fox derivatives
normalized by $a-\epsilon(a)=\sum_x(x-1)D_x^{\mathrm r}(a)$,
the boundaries are
\[
 d_1(e_x^1)=1-\pi(x),\qquad
 d_2(e_y^2)=\sum_xe_x^1\,\pi(D_x^{\mathrm r}(r_y));
\]
see \cite[Chapter~4]{BH1}, \cite[Section~4.2]{FENN}.
The image of $d_2$ is the Magnus image of $R/[R,R]$, with right
action $\overline r\cdot g=\overline{\widetilde g^{-1}r\widetilde g}$
for a lift $\widetilde g\in\Phi$ of $g$ \cite[Corollary~1]{BH1}.

Put $\Lambda=\mathbb Z\Phi$, $\mathfrak f=\ker(\Lambda\to\mathbb Z)$,
and $\mathfrak r=\ker(\Lambda\to\mathbb ZG)$.
Gruenberg's free resolution begins \cite[Theorem~2]{Gru}
\begin{equation}\label{1.2}
\cdots\longrightarrow\mathfrak r^2/\mathfrak r^3
\longrightarrow\mathfrak f\mathfrak r/\mathfrak f\mathfrak r^2
\longrightarrow\mathfrak r/\mathfrak r^2
\longrightarrow\mathfrak f/\mathfrak f\mathfrak r
\longrightarrow\mathbb ZG\longrightarrow\mathbb Z\longrightarrow0.
\end{equation}
The maps $\tau_0=\operatorname{id}$, $\tau_1(e_x^1)=1-x$, and
$\tau_2(e_y^2)=1-r_y$ define a chain comparison through degree two
by the right Fox identity. In general $\tau_2$ need not be an
isomorphism: the $r_y$ normally generate $R$, whereas a free group
basis $B$ of $R$ gives the module basis $\{1-b:b\in B\}$ of
$\mathfrak r/\mathfrak r^2$.

\begin{lemma}\label{l.1.3}
For a contractible presentation $(X\mid Y)$, the abelianized relators
form a basis of $\Phi(X)_{\mathrm{ab}}\cong\mathbb Z^{(X)}$.
If $X$ and $Y$ are finite, then $|X|=|Y|$ and the exponent-sum
matrix is unimodular.
\end{lemma}
\begin{proof}
Since $G=1$ and $K$ is contractible, its cellular boundary
$\mathbb Z^{(Y)}\to\mathbb Z^{(X)}$ is an isomorphism.
Its columns are the abelianized relators.
\end{proof}
In this case $R=\Phi$ and $\mathfrak r/\mathfrak r^2\cong\Phi_{\mathrm{ab}}$,
so $\tau_2$ is an isomorphism.

\subsection{Simplicial presentation as a loop-space model}\label{r1.1}
By definition, a simplicial resolution of a group $ G $ is a free simplicial group $F_{\bullet}$ (i.e. $F_n$ is a free group) such that $\pi_0F_{\bullet}\cong G$ and $\pi_nF_{\bullet}=0$ for $n>0$, this concept is analogous to the Eilenberg - MacLane space $ K (G, 1) $ in the category of simplicial groups. By \textbf{homotopy groups of a simplicial group} $F_{\bullet}$ we understand the homology groups of its \textbf{Moore complex} $(NF_n=\cap^{n-1}_{i=0}Ker\;d_i^n, d_n^n|_{NF_n})$.

The step-by-step construction of a free simplicial resolution starts
with generators in degree zero and adds generators in degree one for
the chosen relators; see \cite[Section~1.2]{MP}. Denote this initial
stage by $P_\bullet=F^{(1)}_\bullet$. It is a full simplicial group,
with no new nondegenerate free generators above degree one. Only
its restriction to degrees $0,1,2$ and its augmentation are displayed:
\begin{equation}\label{2}
\xymatrix@C=3.8em{
 P_2 \ar@<2ex>[r]^{d_0} \ar[r]|{d_1}
     \ar@<-2ex>[r]_{d_2} &
 P_1 \ar@<4ex>[l]^{s_0} \ar@<-4ex>[l]_{s_1}
     \ar@<1.5ex>[r]^{d_0} \ar@<-1.5ex>[r]_{d_1} &
 P_0 \ar@<3.5ex>[l]^{s_0} \ar[r] & G.}
\end{equation}
Here
\[
\begin{aligned}
 P_0&=\Phi(X),\qquad P_1=\Phi(s_0X\sqcup Y),\\
 P_2&=\Phi(s_1s_0X\sqcup s_0Y\sqcup s_1Y),
\end{aligned}
\]
where $s_0s_0=s_1s_0$, and the degree-one faces satisfy
\[
 d_0(s_0x)=d_1(s_0x)=x,\qquad d_0(y)=1,\qquad d_1(y)=r_y.
\]
Higher faces on the degenerate generators are determined by the
simplicial identities. In particular,
$\operatorname{im}(d_1:\ker d_0\to P_0)=R$ and $\pi_0P_\bullet=G$.
The diagram $P_1\rightrightarrows P_0$ with its degeneracy is called
a \textbf{simplicial presentation of $G$}.

The Moore complex of $P_\bullet$ has the form
\[
 \cdots\longrightarrow N_2P\xrightarrow{d_2}N_1P
 \xrightarrow{d_1}N_0P,
\]
where $N_0P=P_0$ and
$N_nP=\bigcap_{i=0}^{n-1}\ker(d_i:P_n\to P_{n-1})$ for $n\geq1$.
Thus
\[
 \pi_1P_\bullet
 =\frac{\ker(d_1:N_1P\to N_0P)}{d_2(N_2P)}.
\]
Generation by degeneracies in degree two does not imply $N_2P=1$.
In general $P_\bullet$ is not yet a resolution of $G$: its positive
homotopy groups have not been killed.

For the standard presentation complex $K=K(X\mid Y)$, choose Kan's
cellular free simplicial model $B_K$ with degree-zero generators $X$
and degree-one generators $Y$, with the attaching words specified
above \cite[Definition~5.2 and Theorem~6.2]{Kan}. There are no higher
nondegenerate generators because $K$ is two-dimensional. With these
choices $P_\bullet$ is this cellular model, and \eqref{2} displays its
restriction to degrees $0,1,2$. Kan's comparison
\cite[Theorem~5.5]{Kan} gives the loop homotopy equivalence with
$GS_1K$, where $S_1K$ is the reduced first Eilenberg subcomplex of the
singular complex $SK$ \cite[Definition~8.3 and Theorem~8.4]{May}.
In particular, as explained below,
\[
 |P_\bullet|\simeq\Omega K,\qquad
 \pi_1P_\bullet\cong\pi_2K.
\]
We retain this model of the presentation; we do not add further
nondegenerate generators to turn it into a resolution of $G$.
No nilpotency assumption on $G$ or $K$ is made here.

Recall that the ``Kan loop group functor'' $G$ assigns to each reduced simplicial set $X$ a free simplicial group $GX$ subject to the ``principal twisted cartesian product'' \cite[Theorem 26.6]{May}
$GX\rightarrow GX \times_{\tau} X \xrightarrow{p} X,$
where $|GX \times_{\tau} X|$ is a contractible space, $p$ is a Kan fibration and therefore $GX$ is a simplicial ``loop space'' for $X$.
By \cite[Theorem 10.10]{GJ}, the geometric realization $|p|$ of the Kan fibration $p$ is a Serre fibration. Then by \cite[Theorem 10.9]{GJ} $|GX|$ is the fiber of 
$|GX|\rightarrow |GX \times_{\tau} X| \xrightarrow{|p|} |X|$
and, therefore, by \cite[Proposition 4.66]{Hut} there is a weak homotopy equivalence $\beta:|GS_1K(X|R)|\simeq \Omega K(X|R)$.
Since $K(X|R)$ is a $CW$-complex, by \cite[Corollary 2]{Milnor} $\Omega K(X|R)$ is homotopy equivalent to a $CW$-complex.
And, as geometric realization is a CW-complex \cite[Theorem 14.1]{May}, by Whitehead's theorem, $\beta$ is a homotopy equivalence. Finally, we have a homotopy equivalence $|B_{K(X|R)}|\simeq \Omega K(X|R)$ and therefore $B_{K(X|R)}$ is a combinatorial model of based loops on $K(X|R)$.

A pro-$ p $-group is a group isomorphic to the inverse limit of finite $ p $-groups. This is a topological group (with inverse limit topology) which is compact and totally disconnected.
For such groups one has a presentation theory similar in many aspects to the combinatorial theory of discrete groups.

We next recall a classical arithmetic square under a nilpotency
hypothesis. This hypothesis applies only to the auxiliary statement
below; later it will be used at the nilpotent stages of the tower.
The general square in Proposition~\ref{general-free-arithmetic} does
not assume that the original space or simplicial group is nilpotent.
For a connected nilpotent space $X$ of finite type, the arithmetic square
of Bousfield and Kan \cite[Chapter~VI, Lemma~8.1]{BK} is a homotopy
pullback.  Write $\rho_0^T:T\to\mathbb Q_\infty T$ for the natural
rational-completion map and
$\kappa=\prod_{p\in\pi}\kappa_p$, where $\pi$ is the set of all primes.
Then the square is
\begin{equation}\label{a2}
\xymatrix{
 X\ar[rrr]^{\kappa}\ar[d]_{\rho_0^X} &&&
 \prod_{p\in\pi}(\mathbb F_p)_\infty X\ar[d]^{\rho_0^{\prod_p(\mathbb F_p)_\infty X}}\\
 \mathbb Q_\infty X\ar[rrr]^{\mathbb Q_\infty(\kappa)} &&&
 \mathbb Q_\infty\bigl(\prod_{p\in\pi}(\mathbb F_p)_\infty X\bigr).}
\end{equation}
Here the vertical arrows are maps, whereas $\mathbb Q_\infty$ denotes a
functor; the bottom arrow is its value on $\kappa$.  On nilpotent spaces,
rational completion agrees up to homotopy with rationalization
\cite[Chapter~V, Section~4.3]{BK}.  The corresponding group constructions
are rational prounipotent completion and pro-$p$ completion.  For the
interpretation of pro-$p$ groups as prounipotent group schemes over
$\mathbb F_p$, see \cite[Section~4.1]{Mikh2023}.
\begin{definition}\label{d4}\cite[2]{Mikh2023}
Let us fix a group $G$ and a field $k$ of zero characteristics, define the prounipotent completion of $G$ as the following universal diagram, in which $\rho$ is a
Zariski dense homomorphism from $G$ to the group of $k$-points
of a prounipotent affine group scheme $G_k^{\wedge}$:
$$\xymatrix @R=0.5cm{
                &         G^{\wedge}_k(k)  \ar[dd]^{\tau}     \\
 G \ar[ur]^{\rho} \ar[dr]_{\chi}                 \\
                &         H(k)              }$$
We require that for each Zariski dense homomorphism
$\chi$ there exists a unique homomorphism $\tau$ of prounipotent groups making the diagram commutative.
When $k=\mathbb{F}_p$, we define the prounipotent completion of $G$ as the pro-$p$-completion, which is a pro-$p$-group $G^{\wedge}_p$ obeys the same universal diagram, where $H(k)$ is a finite $p$-group ($G^{\wedge}_p\cong\varprojlim_{|G/U_{\lambda}|=p^{k_{\lambda}}} G/U_{\lambda}$).
\end{definition}

Bousfield and Kan show \cite[Proposition 4.1]{BK} that for each simplicial reduced space $X$ of finite type one can construct $R_{\infty}$-completion up to homotopy ($R=\mathbb{Q},\mathbb{F}_p$) in three steps:
\begin{description}
\item[i]{Replace $X$ with the so-called ``Kan's loop group'' $GX$ \cite[p.276]{GJ} - a simplicial group that has the homotopy type of ``loops on $X$'';}
\item[ii]{Apply the $R$-prounipotent completions functor dimension-wise, that is, we obtain a simplicial prounipotent group $(GX)^{\wedge}_R$;}
\item[iii]{Take the classifying space $\overline{W}(GX)^{\wedge}_R$ of $(GX)^{\wedge}_R$.}
\end{description}

The Kan loop-group functor $G$ and the classifying-space functor
$\overline W$ form an adjoint pair inducing the equivalence of the
corresponding homotopy categories \cite[Chapter~V, Proposition~6.3]{GJ}.
Let $P_\bullet$ be the free simplicial presentation associated with
$K(X\mid Y)$.  The relation between presentation homotopy groups and
simplicial-group homotopy groups is
\[
 \pi_n K(X\mid Y)\cong\pi_n(\overline WP_\bullet)
 \cong\pi_{n-1}(P_\bullet)\qquad(n\geq2).
\]
In particular, the second homotopy group of the presentation is
$\pi_2K(X\mid Y)\cong\pi_1(P_\bullet)$.  The index shift concerns
$\overline WP_\bullet$, not the realization $|P_\bullet|$, whose homotopy
groups agree with those of $P_\bullet$ without a shift.  The same
convention defines the second homotopy group of a prounipotent
presentation \cite[Section~4.5]{Mikh2023}.

\subsection{Outline of the argument}\label{intro:strategy}
For a finite presentation, write
$\widehat P_{R,\bullet}=(P_\bullet)^{\wedge}_R$ for the degreewise
rational or pro-$p$ completion of its simplicial model.  Then
\[
 R_\infty K(X\mid Y)\simeq\overline W((P_\bullet)^{\wedge}_R),
 \qquad R=\mathbb Q,\mathbb F_p.
\]

For a finite subcontractible presentation, the ambient contractible
complex supplies a primitive independent family of abelianized
relators (Lemma~\ref{lem:relator-independence}).  After rational or
pro-$p$ completion, these relators extend to a free topological basis
$Y_1\sqcup Z$, and the quotient group is $F_R(Z)$, the free
prounipotent group read on rational points or the free pro-$p$ group,
respectively.  The canonical
relator-basis condition of Definition~\ref{d7} makes the cohomological
asphericity criterion applicable to this presentation, including the
non-minimal case.  The Brown--Loday identity identifies the resulting
low-dimensional obstruction with $\pi_1\widehat P_{R,\bullet}$.
To compute the higher homotopy groups we use the filtration, whose
initial input can be read directly from the chosen words.

Put $A_\bullet=\operatorname{Ab}_R(\widehat P_{R,\bullet})$.
Its normalized complex is
\[
 0\longrightarrow R^{Y_1}
 \xrightarrow{\,e_y\mapsto\operatorname{ab}_R(r_y)\,}
 R^{X_1}\longrightarrow0,
\]
in degrees one and zero.  The differential is a split injection, so
$A_\bullet$ is simplicially homotopy equivalent to the constant
module $R^Z$.  The closed lower central series in characteristic zero
and the Zassenhaus filtration in characteristic $p$ have graded
pieces $L_s^R(A_\bullet)$, where $L_s^R$ is the ordinary or restricted
free Lie power.  Thus the associated Adams--Quillen spectral sequences
have
\[
 E^1_{n,s}=\pi_nL_s^R(A_\bullet)=0
 \qquad(n>0,\ s\geq1).
\]
A dimensionwise Lie functor preserves the
simplicial homotopy to the constant module.  Theorem~\ref{main}
then verifies convergence on the successive quotient fibrations:
the graded vanishing gives vanishing at every finite filtration
stage, and Milnor's exact sequence carries it to the complete
simplicial group.

The augmentation
$\widehat P_{R,\bullet}\to F_R(Z)$ is then a free simplicial resolution.
Keune's comparison theorem lifts maps between the augmented groups
to maps of resolutions, uniquely up to simplicial homotopy.  In the
prounipotent category, Proposition~\ref{prop:keune-finite-type} and
Corollary~\ref{cor:prounipotent-model-comparison} therefore realize
\[
 \widehat P_{\mathbb Q,\bullet}\simeq
 \operatorname{Const}(F_{\mathbb Q}(Z))
\]
by morphisms and homotopies in that category.  The continuous lifting argument, including infinite convergent bases
and the pro-$p$ case, is given in
\cite[Proposition~2.3 and Theorem~4.2]{MikhContinuous}.

The arithmetic square assembles the completions through their
common nilpotent quotients.  For any free discrete simplicial group
$P_\bullet$ of finite type, put
$N_{s,\bullet}=P_\bullet/\gamma_sP_\bullet$.  In each degree these
are finitely generated free nilpotent groups.  Their arithmetic
squares have an adelic entry given by a restricted product of
$p$-adic unipotent groups relative to their pro-$p$ subgroups
(Proposition~\ref{asnfn}).  Rational points map diagonally into this
entry.  Arithmetic factorization makes the corresponding square
of classifying spaces a connected homotopy pullback.  Taking the
homotopy inverse limit of these squares gives
\[
 \xymatrix{
 \widehat P_{\mathbb Z,\bullet}\ar[r]\ar[d] &
 \prod_p\widehat P_{p,\bullet}\ar[d]\\
 \widehat P_{\mathbb Q,\bullet}\ar[r] & \mathcal A(P)_\bullet,}
\]
where $\widehat P_{\mathbb Z,\bullet}=\varprojlim_sN_{s,\bullet}$
and the other completed groups are formed degreewise.  The
lower-right entry is the explicit simplicial group
\[
 \mathcal A(P)_\bullet=
 \varprojlim_s\varinjlim_{S\subset\pi,\ S\text{ finite}}
 \left(
 \prod_{p\in S}((N_{s,\bullet})_p^{\wedge})_0
 \times\prod_{p\notin S}(N_{s,\bullet})_p^{\wedge}
 \right).
\]
The finite exceptional set of primes may depend on the nilpotent
stage $s$.  This feature accounts for the order of the two limits
and is examined in Section~\ref{adelic-explicit}.  The square is a
homotopy pullback by Proposition~\ref{general-free-arithmetic};
its construction uses nilpotency at the quotient stages and imposes
no nilpotency condition on $\pi_0P_\bullet$.  Applying the derived
classifying-space functor identifies its upper-left entry with the
integral Bousfield--Kan completion of $\overline WP_\bullet$.

For applications to one-relator presentation complexes, see
\cite{MikhAdelic}: rational and adelic asphericity allow the arithmetic
square to decompose each higher homotopy group of integral completion
as a product of abelian pro-$p$ groups, namely the corresponding
homotopy groups of the $p$-adic completions.

For a subcontractible presentation the same complementary words
can be chosen at every nilpotent stage.  They give compatible
contractions
\[
 P_\bullet/\gamma_sP_\bullet\simeq
 \operatorname{Const}\bigl(\Phi(Z)/\gamma_s\Phi(Z)\bigr).
\]
Consequently all four entries of the arithmetic square can be
identified by taking limits of explicit groups.  Theorem~\ref{c4}
gives, in particular,
\[
 R_\infty K(X_1\mid Y_1)\simeq K(F_R(Z),1),
 \qquad R=\mathbb Q,\mathbb F_p,\mathbb Z,
\]
with $F_{\mathbb Z}(Z)=\varprojlim_s\Phi(Z)/\gamma_s\Phi(Z)$.
The higher homotopy groups vanish at each stage, and the group towers
are surjective.

Section~\ref{s1.0} develops the nilpotent arithmetic square and the
homotopy-limit constructions used to assemble its stages.
Section~\ref{s5} establishes independence of the relators, proves
the field-completion theorem by the filtration argument, and gives
the comparison of free simplicial prounipotent resolutions.
Section~\ref{s6} constructs the general noncommutative arithmetic
square, describes its adelic entry, and carries out the integral
completion calculation for subcontractible presentations.

\section{Arithmetic squares for nilpotent spaces and their homotopy (co)limits}\label{s1.0}
\setcounter{theorem}{0}
We start with the "arithmetic square" of abelian groups \cite[Proposition 1.18]{Su1}
\begin{proposition}\label{su}
There is the following diagram of abelian groups
$$\xymatrix{\mathbb{Z}  \ar[rr]^{\prod_{p\in\pi} \kappa_p}\ar[d]_{\otimes \mathbb{Q}} && \prod_{p\in\pi}  \mathbb{Z}_p\ar[d]^{\otimes\mathbb{Q}} \\
\mathbb{Q}\ar[rr]^{\prod_{p\in\pi} \kappa_p} &&\widetilde{\prod}_{p\in\pi} \mathbb{Q}_p,}$$
where $\pi$ is the set of prime numbers, $\kappa_p$ is the $p$-adic
completion map, and $\widetilde{\prod}_{p\in\pi}\mathbb Q_p$ is the
restricted product: all but finitely many components belong to
$\mathbb Z_p$.
\end{proposition}
\begin{proof}
It suffices to prove exactness of
$$0\rightarrow \mathbb{Z}\rightarrow \prod_{p\in\pi}  \mathbb{Z}_p\oplus\mathbb{Q}  \xrightarrow{\otimes \mathbb{Q}-\prod_{p\in\pi} \kappa_p} \widetilde{\prod}_{p\in\pi} \mathbb{Q}_p\rightarrow 0.$$
For any $p$-adic number $r_p\in\mathbb{Q}_p$ there is a minimal $n\in\mathbb{N}$ such that $r_p=1/p^n\cdot \xi_p,$ where $\xi\in\mathbb{Z}_p$. Since $\mathbb{Z}_p$ is the $p$-adic completion of $\mathbb{Z}$ it follows that elements of $\mathbb{Z}_p$ can be written as series of the form $\sum_{i=0}^{\infty} a_i\cdot p^i$, where $a_i$ are non-negative integers not exceeding $p-1$. It follows that for each $p$-adic number $r_p$ there is a rational number $m/p^n$, such that $r_p-m/p^n$ is a $p$-adic integer and therefore the sequence is right exact. It is also left exact since a rational number is a $p$-adic integer if and only if it is an integer.
\end{proof}
The $p$-adic completion and the rationalization can be extended to the category of groups.

Let $G$ be a group and let $p$ be a prime.  If $I$ is the augmentation
ideal of $\mathbb F_p[G]$, the Zassenhaus filtration is
\[
 \mathcal M_s(G)=\{g\in G\mid g-1\in I^s\}.
\]
Equivalently, the dimension-subgroup formula gives
\[
 \mathcal M_s(G)=\prod_{ip^j\geq s}\gamma_i(G)^{p^j},
 \qquad i\geq1,\quad j\geq0;
\]
see \cite{Rec,Cur} and \cite[Introduction and Section~8]{ChapmanEfrat2016}.
For pro-$p$ groups we use the completed group algebra, closed powers
of its augmentation ideal, and the closure of the displayed product.
Thus the augmentation-ideal definition here and the Zassenhaus
filtration used in the field-completion proof describe the same
filtration.  To avoid a terminological ambiguity, we distinguish it
from the lower $p$-central series
\[
 P_1G=G,\qquad P_{s+1}G=(P_sG)^p[P_sG,G],\qquad
 P_sG=\prod_{i=1}^s\gamma_i(G)^{p^{s-i}}.
\]
The latter product formula is the one discussed in
\cite[Introduction and Corollary~7.5]{ChapmanEfrat2016}, with reference
to \cite[Proposition~3.8.6]{NSW}.  These filtrations need not agree
term by term: for $G=\mathbb Z$ and $p=2$,
$\mathcal M_4(G)=4\mathbb Z$, whereas $P_4G=8\mathbb Z$.

Now we can define the $p$-adic completion of a group $G$ by the rule as follows $$G^{\wedge}_p = \varprojlim_s G/\mathcal{M}_s.$$
in the special case, when $G$ is finitely generated, then  $G^{\wedge}_p$ is a pro-$p$-completion of $G$ and thus, if $G$ is abelian, then $G^{\wedge}_p\cong G\otimes \mathbb{Z}_p$.

For a nilpotent group $G$, its rationalization $G_0$ is its Malcev
completion: the universal map to a uniquely divisible nilpotent group.
For an arbitrary group we instead distinguish the rational pronilpotent
completion
\[
 \widehat G_{\mathbb Q}^{\mathrm{nil}}
 =\varprojlim_s(G/\gamma_sG)_0.
\]
Let $J$ be the augmentation ideal of $\mathbb Q[G]$.  Quillen's
complete Hopf-algebra description identifies this completion with
the group-like elements of
\[
 \widehat{\mathbb Q[G]}=\varprojlim_m\mathbb Q[G]/J^m;
\]
see \cite[Appendix~A, Section~3]{Qui5}.  The Hopf structure is understood
on the completion, with completed tensor products.  The individual
algebra quotients by $J^m$ are not being asserted to be Hopf quotients.
The augmentation filtration and the nilpotent description give
compatible models of the same group completion.  Indeed, every map
to a rational nilpotent group of class at most $s$ kills
$\gamma_{s+1}G$ and factors uniquely through $(G/\gamma_{s+1}G)_0$.
Consequently these rationalized nilpotent quotients suffice to compute
the universal inverse limit, and all the identifications are natural
in $G$.  In the free finitely generated case used below, the Magnus
expansion identifies $\widehat{\mathbb Q[G]}$ with noncommutative power
series; its primitive elements form the completed free Lie algebra,
and truncation by bracket length gives the same nilpotent tower.
This group completion should not be confused with localization in
the category of all groups.

We recall the localization terminology only in the category of nilpotent
groups, where it is needed for the arithmetic square
\cite[Definition~1.1]{HMR}.
\begin{definition}\label{hmr}
Let $P$ be a set of primes.  A nilpotent group $G$ is \emph{$P$-local}
if $g\mapsto g^q$ is bijective for every prime $q\notin P$.
A \emph{$P$-localization} of $G$ is a homomorphism
$\kappa_P:G\to G_P$ to a $P$-local nilpotent group such that
precomposition induces a bijection
\[
 \operatorname{Hom}(G_P,H)\xrightarrow{\ \sim\ }
 \operatorname{Hom}(G,H)
\]
for every $P$-local nilpotent group $H$.
\end{definition}
Localization and completion are different operations.  Already for the
additive group of integers,
\[
 \mathbb Z_{\{p\}}=\mathbb Z_{(p)}
   =\{a/b\in\mathbb Q:p\nmid b\},
 \qquad \mathbb Z_p^{\wedge}=\mathbb Z_p.
\]
The former inverts primes other than $p$; the latter is the inverse
limit of the quotients $\mathbb Z/p^m\mathbb Z$.  The first group is
countable and the second uncountable.  Rationalization is the special
case $P=\emptyset$, denoted $G_0$, and coincides with Malcev completion
for nilpotent groups.  At the space level, Bousfield--Kan rational
completion agrees with rationalization for nilpotent spaces, but this
identification cannot be assumed for arbitrary spaces
\cite[Sections~7--8]{IvanovRat2021}.  We therefore retain separate
notation for localization, pro-$p$ completion, group pronilpotent
completion, and Bousfield--Kan completion.

It is proved in \cite[Lemma 8.2, Ch. VI]{BK}, \cite[Theorem I.3.7]{HMR} (here the proof is very similar to that of Proposition \ref{su})that any finitely generated nilpotent group is a pullback of the "arithmetic square" as follows
\begin{theorem}\cite[Lemma 8.2, Ch. VI]{BK} \label{asn}
If $G$ is a finitely generated nilpotent group and $P$ is a set of primes, then $G_P$ is a pullback of the arithmetic square below.  In this diagram $G_p$ denotes the pro-$p$ completion $G_p^{\wedge}$, whereas $G_P$ denotes localization at $P$.
$$\xymatrix{G_P  \ar[rrr]^{\prod_{p\in P} \kappa_p}\ar[d]_{\kappa_0} &&& \prod_{p\in P}  G_p\ar[d]^{\kappa_0} \\
 G_0 \ar[rrr]^{\kappa_0(\prod_{p\in P} \kappa_p)} &&&(\prod_{p\in P} G_p)^{\wedge}_0,}$$
Moreover, every element $u\in ( \prod_{p\in P} G_p)_0$ can be expressed as $u=vw$, where $v$ (resp. $w$) is in the image of 
$G_0$ (resp. $\prod_{p\in P} G_p$). If we assume the set $P=\pi$, then we obtain a pull-back presentation of a finitely generated nilpotent group in the "arithmetic square".
\end{theorem}

By a $\lambda$-nilpotent group we mean a group $G$ with
$\gamma_{\lambda+1}G=1$, where $\gamma_1G=G$ and
$\gamma_{j+1}G=[G,\gamma_jG]$.  The following proposition describes the lower-right corner for
finitely generated nilpotent groups.
\begin{proposition}\label{asnfn}
Let $G$ be a finitely generated nilpotent group.  Put $N_p=G_p^{\wedge}$,
$N=\prod_pN_p$, and let $U$ be the unipotent algebraic group over
$\mathbb Q$ associated with its Malcev completion, so that
$G_0=U(\mathbb Q)$.  Write $U_p=U\otimes_{\mathbb Q}\mathbb Q_p$ and
$j_p:N_p\to U_p(\mathbb Q_p)$ for the natural map.  Then:
\begin{description}
\item[(1)] The rationalization of the underlying nilpotent group $N_p$
is naturally $(N_p)_0\cong U_p(\mathbb Q_p)$.
\item[(2)] For finite sets $S$ of primes, put
\[
 A_S=\prod_{p\in S}U_p(\mathbb Q_p)\times\prod_{p\notin S}N_p,
 \qquad A=\varinjlim_{S}A_S.
\]
For $S\subseteq T$, the transition map applies $j_p$ in the newly added
coordinates $p\in T\setminus S$ and is the identity elsewhere.  There
is a natural isomorphism $N_0\cong A$.  Equivalently, with
$K_p=j_p(N_p)$,
\[
 A\cong\prod_p'\bigl(U_p(\mathbb Q_p),K_p\bigr),
\]
the group of tuples whose coordinates belong to $K_p$ for all but
finitely many primes.
\item[(3)] Let $\kappa:G\to N$ be the product of the completion maps,
and let $\delta=\kappa_0:G_0\to N_0\cong A$ be its rationalization.
The arithmetic square is the pullback
\begin{equation}\label{fn}
\xymatrix{
 G\ar[rr]^{\kappa}\ar[d]_{\kappa_0} &&N\ar[d]^{\kappa_0}\\
 G_0\ar[rr]^{\delta} &&A.}
\end{equation}
In restricted-product coordinates,
$\delta(u)=(u_p)_p$, where $u_p$ is the image of $u\in U(\mathbb Q)$
under scalar extension to $\mathbb Q_p$.  This is a diagonal map over
\emph{all} primes.  By contrast, the map of an individual factor
$N_p\to A$ has only its $p$-coordinate possibly nontrivial and factors
through $j_p$.
\item[(4)] If $G$ is a free nilpotent group on finitely many generators,
all four arrows in \eqref{fn} are injective.  Thus, inside $A$,
\[
 G\cong\delta(G_0)\cap\kappa_0(N).
\]
\end{description}
Here $U_p(\mathbb Q_p)$ is also denoted
$G_{\mathbb Q_p}^{\wedge}(\mathbb Q_p)$; the notation distinguishes the
algebraic group from its group of rational points.
\end{proposition}
\begin{proof}
(1) The Malcev group of a finitely generated nilpotent group is
finite-dimensional.  The comparison of unipotent completions identifies
the $p$-adic unipotent completion with $U_p$; see
\cite[Appendix~A, Theorem~A.6 and Lemma~A.7]{HM2003}.
The map $j_p$ has torsion kernel, and its image $K_p$ is a compact open
subgroup of $U_p(\mathbb Q_p)$.  These facts can also be seen by taking
Malcev coordinates on $G/\operatorname{Tor}(G)$: the pro-$p$ completion
replaces integral coordinates by $p$-adic integral coordinates.  The
finite torsion subgroup contributes only a finite $p$-group to the
kernel; compare \cite[Appendix~A, Theorem~A.3 and Lemma~A.5]{HM2003}.
Every $v\in U_p(\mathbb Q_p)$ has a power in $K_p$.  Indeed,
$v^{p^m}=\exp(p^m\log v)$ tends to the identity, hence lies in the open
subgroup $K_p$ for large $m$.  Since $U_p(\mathbb Q_p)$ is uniquely
divisible and nilpotent, the rationalization criterion
\cite[Appendix~A, Corollary~3.8]{Qui5} proves (1).

(2) The nilpotency classes of all $A_S$ are bounded by that of $G$.
For a pro-$p$ group the $n$th-power map is bijective when $(n,p)=1$:
this holds on every finite $p$-group and passes to the inverse limit.
Given $a\in A_S$ and $n\geq1$, enlarge $S$ by the primes dividing $n$.
In the enlarged coordinates take roots in the uniquely divisible groups
$U_p(\mathbb Q_p)$; outside them take the unique roots in $N_p$.
This proves existence of $n$th roots in $A$.  If two elements have the
same $n$th power in the colimit, their equality of powers holds at some
finite stage; enlarging that stage by the primes dividing $n$ proves
uniqueness in exactly the same way.  Thus $A$ is uniquely divisible.

Let $j:N=A_{\emptyset}\to A$ be the canonical map.  For $a\in A_S$,
part (1) supplies a positive integer $n$ such that every coordinate in
$S$ of $a^n$ lifts to $N_p$.  The remaining coordinates already lie in
$N_p$, so $a^n\in j(N)$.  If $b\in\ker j$, its image is trivial at
some finite stage $A_S$.  Hence $b_p=1$ outside $S$, while each $b_p$
inside $S$ belongs to the torsion kernel of $j_p$.  A common multiple
of their orders annihilates $b$.  Conversely, all torsion elements map
to the torsion-free group $A$.  The same rationalization criterion now
gives $N_0\cong A$.

The coordinate maps identify $A$ with the displayed restricted product.
Surjectivity follows by lifting every coordinate in $K_p$ outside a
finite exceptional set.  For injectivity, note that
$\operatorname{Tor}(G)$ is finite, so the kernels of $j_p$ are trivial
outside a fixed finite set of primes.  An element with all coordinates
trivial therefore becomes trivial after adjoining that finite set to
$S$.  This also explains why the transition maps need not be embeddings
when $G$ has torsion.

(3) Every $u\in G_0$ has a positive power $u^n$ in the image of $G$.
For $p\nmid n$, the unique $n$th root of the image of $u^n$ in $N_p$
maps to $u_p$.  Therefore $u_p\in K_p$ except possibly at the finitely
many primes dividing $n$, and $(u_p)_p$ belongs to the restricted
product.  Its map from $G_0$ agrees with $j\kappa$ on $G$, hence equals
$\delta$ by the universal property of rationalization.  The pullback
assertion is Theorem~\ref{asn} with all primes.  Notice that an embedding
into a single coordinate of $A$ would not give this diagonal map.

(4) Write $G=\Phi(X)/\gamma_{\lambda+1}\Phi(X)$ with $X$ finite.
The Hall basic commutators of weights at most $\lambda$ give unique
integral coordinates on $G$.  The same coordinates are in $\mathbb Z_p$
for $N_p$, in $\mathbb Q$ for $G_0$, and in $\mathbb Q_p$ for
$U_p(\mathbb Q_p)$.  The inclusions of these coefficient rings show that
$G\to G_0$, $G\to N_p$, and $j_p$ are injective.  The coordinatewise
map $N\to A$ and the diagonal map $G_0\to A$ are consequently
injective, proving (4).

The associated-graded argument behind these coordinates
is as follows.  Witt's theorem gives
\[
 \operatorname{gr}_{\gamma}\Phi(X)\cong L_{\mathbb Z}(X).
\]
In characteristic zero, the Lie algebra of the free prounipotent group
is the \emph{completed} free Lie algebra
\[
 \operatorname{Lie}F_k(X)=\widehat L_k(X)
 =\prod_{m\geq1}L_{k,m}(X).
\]
Its degree filtration has associated graded
$\bigoplus_{m\geq1}L_{k,m}(X)=L_k(X)$.  The exponential and logarithm
identify the group with this completed Lie algebra equipped with the
Baker--Campbell--Hausdorff multiplication, not with its additive group.
On each successive lower-central factor the multiplication becomes additive.
Since each $L_{\mathbb Z,m}(X)$ has a Hall basis, extension of scalars gives
\[
 \operatorname{gr}_{\gamma,m}F_k(X)(k)
 \cong L_{k,m}(X)
 \cong L_{\mathbb Z,m}(X)\otimes_{\mathbb Z}k.
\]
Here lower-central subgroups on the prounipotent side are closed;
see \cite[Chapter~IV, Sections~6--7]{Se} and
\cite[Appendix~A, Section~3]{Qui5}.
For the free pro-$p$ group, the closed ordinary lower-central quotients
are likewise
\[
 \operatorname{gr}_{\gamma,m}F_p(X)
 \cong L_{\mathbb Z,m}(X)\otimes_{\mathbb Z}\mathbb Z_p
\]
\cite[Propositions~3.2.2 and~3.2.5]{ZR}.  This is the ordinary closed
lower central series, not the Zassenhaus filtration.  Truncating these
identifications at weight $\lambda$ yields the free nilpotent case
used above.
\end{proof}

Bousfield and Kan show \cite[Proposition 4.1]{BK} that for each simplicial reduced space $X$ of finite type one can construct $R_{\infty}$-completion up to homotopy (we meet $R=\mathbb{Z}, \mathbb{Q},\mathbb{F}_p$) in three steps:
\begin{description}
\item[i]{Replace $X$ with the so-called ``Kan's loop group'' $GX$ \cite[p.276]{GJ} - a simplicial group that has the homotopy type of ``loops on $X$'';}
\item[ii]{Apply the $R$-prounipotent completions functor dimension-wise, that is, we obtain a simplicial prounipotent group $(GX)^{\wedge}_R$;}
\item[iii]{Take the classifying space $\overline{W}(GX)^{\wedge}_R$ of $(GX)^{\wedge}_R$.}
\end{description}

\begin{theorem}\cite[Chapter~V, Section~3]{BK}\label{q}
If $X$ is pointed, connected and nilpotent, then
$X\to\mathbb Q_\infty X$ is a rational homology equivalence and
\[
 \pi_1(\mathbb Q_\infty X)\cong(\pi_1X)_0,\qquad
 \pi_n(\mathbb Q_\infty X)\cong\pi_n(X)\otimes\mathbb Q
 \quad(n\geq2).
\]
The first formula uses nilpotent group rationalization; it is not a
formula involving the tensor product of a nonabelian group.
\end{theorem}
\begin{theorem}\cite[Chapter~VI, Section~5]{BK}\label{p}
Let $X$ be pointed, connected and nilpotent, with $\pi_1X$ finitely
generated and $\pi_nX$ finitely generated abelian for all $n\geq2$.
Then $X\to(\mathbb F_p)_\infty X$ is a mod-$p$ homology equivalence and
\[
 \pi_1((\mathbb F_p)_\infty X)\cong(\pi_1X)_p^{\wedge},\qquad
 \pi_n((\mathbb F_p)_\infty X)\cong\pi_n(X)\otimes\mathbb Z_p
 \quad(n\geq2).
\]
\end{theorem}
\noindent\emph{Why finite generation occurs in the second theorem.}
For a general abelian group $A$, derived $p$-completion involves
\[
 L_0^pA=\operatorname{Ext}^1_{\mathbb Z}(\mathbb Z/p^{\infty},A),
 \qquad L_1^pA=\operatorname{Hom}_{\mathbb Z}(\mathbb Z/p^{\infty},A).
\]
For example, for simply connected $X$ the general formula is
\[
 0\longrightarrow L_0^p(\pi_nX)
 \longrightarrow\pi_n((\mathbb F_p)_\infty X)
 \longrightarrow L_1^p(\pi_{n-1}X)\longrightarrow0
 \qquad(n\geq2);
\]
see \cite[Chapter~VI, Proposition~5.1]{BK}.  For finitely generated $A$,
$L_1^pA=0$ and $L_0^pA=A\otimes\mathbb Z_p$, which explains the simple
formula above.  Without this hypothesis additional homotopy can appear:
for $A=\mathbb Z/p^{\infty}$ one has $L_0^pA=0$ and
$L_1^pA\cong\mathbb Z_p$, so the completion of $K(A,2)$ is
$K(\mathbb Z_p,3)$.  Finite generation is thus a sufficient condition
for the tensor-product formula, not a requirement for defining
completion itself.  No corresponding finiteness hypothesis is needed
for the rational formula in Theorem~\ref{q}.

Countability alone does not suffice.  For the countable free abelian group
$A=\bigoplus_{j\geq1}\mathbb Z$, ordinary $p$-completion gives
\[
 \widehat A_p=\varprojlim_m A/p^mA
 =\left\{(a_j)\in\prod_{j\geq1}\mathbb Z_p:
       a_j\longrightarrow0\text{ $p$-adically}\right\}
 \supsetneq\bigoplus_{j\geq1}\mathbb Z_p=A\otimes\mathbb Z_p;
\]
the sequence $(p^j)_{j\geq1}$ witnesses strictness.
Conversely, under the hypotheses of Theorem~\ref{p}, writing
$\pi_nX\cong\mathbb Z^{r_n}\oplus T_n$ with $T_n$ finite yields
$\pi_n((\mathbb F_p)_\infty X)\cong
\mathbb Z_p^{r_n}\oplus(T_n)_{(p)}$ for $n\geq2$.
Thus these completed groups are finite when $r_n=0$ and have cardinality
$2^{\aleph_0}$ when $r_n>0$.  Finite type controls the completion
formula; it does not imply countability of the completed homotopy groups.

\subsection{Homotopy limits and colimits}\label{Holim} 
We assume that the reader is familiar with the notions of homotopy limits and colimits appropriate references to our taste are \cite{MV}, \cite{Hir2}.
However we remind some notations for instance.

We call $X$ a $I$-diagram of simplicial sets if $X$ is a functor $X: I\rightarrow SSets$ from the small category $I$ to the category $SSets$ of punctured simplicial sets
and denote the category of such diagrams by $Ssets^I$. We also need the notions of overcategory $(I\downarrow i)$ and of undercategory $(i\uparrow I)$, where $i\in Ob(I)$ \cite[2.1]{Hir2} and their classifying spaces (i.e. nerves of this categories $B(I\downarrow i)$, $B(i\uparrow I)$).
One have $X_0, X_1 : I^{op}\times I\rightarrow Ssets$ - two bifunctors $X_0(i,j)=hom(B(I\downarrow i), X(j))$ and $X_1(i,j):B(I\uparrow I)\times X(j)$, where $hom(A,B)$ is the simplicial mapping space - the internal $Hom$ functor in the category of simplicial sets (i.e. $hom_{SSets}(A,B)$ is also a simplicial set) given for simplicial sets $A,B$ by the formula $hom(A,B)_n=Hom_{SSets}(A\times \Delta^n,B)$.
Now the homotopy limit and colimit are the end and coend of $X_0$ and $X_1$ respectively
$$holim_I X=\int_I X_0\qquad\qquad \text{and}\qquad\qquad hocolim_I X=\int^I X_1.$$
Unravelling the end formula \cite[7.4.6, 7.4.11]{MV}, we view
\[
 \operatorname{holim}_I X\subseteq
 \prod_i\operatorname{hom}(B(I\downarrow i),X(i)).
\]
An $n$-simplex is a family $\xi=\{\xi_i\}_{i\in I}$ such that, for
every morphism $i\to j$, the following diagram commutes; the horizontal
maps are induced by this morphism and the identity of $\Delta^n$:
$$\xymatrix{B(I\downarrow i)\times \Delta^n  \ar[r]\ar[d]_{\xi_i} & B(I\downarrow j)\times \Delta^n\ar[d]_{\xi_j} \\
X(i)\ar[r] &X(j)}$$
commutes and therefore $holim_I X$ can be equivalently defined as a simplicial set of natural transformations of functors 
$holim_I X = hom_{SSets^I} (B(I\downarrow \:),X)$, where 
$hom_{SSets^I} (B(I\downarrow\:),X)_n=hom_{SSets^I}(B(I\downarrow\:)\times \Delta^n,X)$ is a simplicial function complex in the simplicial model category $SSets^I$ \cite[I.5]{GJ}.
The simplicial mapping-space functor is compatible with the total singular-space functor \cite[I.1]{GJ}. In particular there is a natural isomorphism of the spaces of $n$-simplexes $$(Sing(X^K))_n\cong ((Sing\; X)^K)_n,$$
where $X^K$ is $CGHS$- the compactly generated Hausdorf space  of continuous functions from $|K|$- the geometric realization of $K$ to $X$ and $Sing\; X$ is the total singular simplicial space of $X$. Thus  in $SSets$ we have a natural isomorphism $Sing(X^K)\cong (Sing\;X)^K$.
As a consiquence the similar notions of homotopy limits and colimits for diagrams of topological spaced $T: I\rightarrow CGHS$ 
$$holim_I T=\int_I Hom_{CGHS}(|B(I\downarrow i)|, T(j)) \qquad\qquad hocolim_I T=\int^I |B(i\uparrow I)^{op}|\times T(j)$$ 
are compatible with simplicial ones. For instance, there are natural isomorphisms of topological spaces and simplicial sets respectively \cite[14.2]{Hir2} $|hocolim_I\;X|\cong hocolim|X|$, $Sing(holim\; X)\cong holim(Sing(X)))$ as well as natural weak equivalences for objectwise fibrant and cofibrant diagrams respectively $|holim\; X|\simeq holim|X|$ and $hocolim(Sing\; X)\simeq Sing(hocolim\; X)$.

We refer the reader to \cite{MV} and \cite{Hir2} regarding homotopy inverse (co)limits. Our basic examples:
\begin{example}[Homotopy pullback]
Let $I_{pb}$ be the category $\{1\}\to\{1,2\}\leftarrow\{2\}$,
and consider a diagram
\[
 X_1\xrightarrow{f_1}X_{12}\xleftarrow{f_2}X_2
\]
of Kan complexes.  For an arbitrary diagram of simplicial sets, first
apply a functorial objectwise fibrant replacement.

In the end formula for the homotopy limit,
$N(I_{pb}\downarrow\{1\})=N(I_{pb}\downarrow\{2\})=\Delta[0]$,
whereas $K=N(I_{pb}\downarrow\{1,2\})$ consists of two edges with a
common terminal vertex.  Although $|K|$ is homeomorphic to an interval,
$K$ is not isomorphic to $\Delta[1]$.  In particular, this homeomorphism
does not identify the simplicial mapping spaces $X_{12}^{K}$ and
$X_{12}^{\Delta[1]}$.  Likewise, the unit
$X\to\operatorname{Sing}|X|$ is a weak equivalence, not in general an
isomorphism of simplicial sets.

Instead, use the standard path-object factorization of the diagonal
\[
 X_{12}\xrightarrow{\simeq}X_{12}^{\Delta[1]}
 \xrightarrow{(\operatorname{ev}_0,\operatorname{ev}_1)}
 X_{12}\times X_{12}.
\]
The first map assigns the constant path, and the second is a Kan
fibration, since $X_{12}$ is fibrant and
$\partial\Delta[1]\hookrightarrow\Delta[1]$ is a cofibration; see
\cite[Chapter~I, Sections~5 and~9]{GJ}.  Here the cotensor uses the
unpointed interval; for a pointed target its base point is the constant
map to the base point.
Define
\[
 P=(X_1\times X_2)\times_{X_{12}\times X_{12}}
 X_{12}^{\Delta[1]}.
\]
Thus, degreewise, $P$ consists of triples $(x_1,x_2,\gamma)$ satisfying
$\operatorname{ev}_0\gamma=f_1(x_1)$ and
$\operatorname{ev}_1\gamma=f_2(x_2)$.  The path-object construction of
homotopy pullbacks gives a natural weak equivalence (or a natural zigzag
of weak equivalences, depending on the chosen homotopy-limit model)
\[
 P\simeq\operatorname{holim}_{I_{pb}}X;
\]
see \cite[Example~8.2.8]{MV}.  We henceforth use $P$ as our model for
$X_1\times^h_{X_{12}}X_2$.  In particular, the following square is an
ordinary pullback defining this model:
\begin{equation}\label{hp}
\xymatrix{
 P\ar[r]\ar[d]_{\operatorname{pr}} &
 X_{12}^{\Delta[1]}\ar[d]^{(\operatorname{ev}_0,\operatorname{ev}_1)}\\
 X_1\times X_2\ar[r]^{f_1\times f_2} & X_{12}\times X_{12}.}
\end{equation}
The projection $\operatorname{pr}$ is a Kan fibration.  If the spaces
and maps are pointed, its fiber over the pair of base points is the
based loop-space model $\Omega X_{12}$.

The following corollary is due to Dyer--Roitberg.
\begin{corollary}\label{dr}
For a diagram of pointed Kan complexes and pointed maps as above, there
is a fiber sequence
\[
 \Omega X_{12}\longrightarrow P
 \xrightarrow{\operatorname{pr}}X_1\times X_2.
\]
With compatible choices of base points, its homotopy exact sequence has,
for $n\geq2$, the segments
\[
 \cdots\longrightarrow\pi_{n+1}(X_{12})
 \longrightarrow\pi_n(P)
 \longrightarrow\pi_n(X_1)\times\pi_n(X_2)
 \xrightarrow{\ f_{1*}-f_{2*}\ }\pi_n(X_{12})
 \longrightarrow\pi_{n-1}(P)\longrightarrow\cdots,
\]
where the loop-space identification is chosen to give the displayed
sign convention.  The low-dimensional end is
\[
 \pi_1(P)\longrightarrow\pi_1(X_1)\times\pi_1(X_2)
 \longrightarrow\pi_1(X_{12})\longrightarrow\pi_0(P)
 \longrightarrow\pi_0(X_1)\times\pi_0(X_2).
\]
This end is interpreted in the usual sense of the homotopy exact
sequence of a fibration, with pointed sets and the fundamental-group
action where appropriate.  In particular, the map from
$\pi_1(X_1)\times\pi_1(X_2)$ to $\pi_1(X_{12})$ is not in general a
group homomorphism, and additive difference notation is not used there.
\end{corollary}
\begin{lemma}
\label{degreewise-pullback}
In Corollary~\ref{dr}, suppose that $X_1,X_2,X_{12}$ are connected,
and put
\[
 \delta_n=f_{1*}-f_{2*}:\pi_nX_1\oplus\pi_nX_2\longrightarrow\pi_nX_{12}
 \qquad(n\geq2).
\]
There are natural short exact sequences
\[
 0\longrightarrow\operatorname{coker}\delta_{n+1}
 \longrightarrow\pi_nP\longrightarrow\ker\delta_n
 \longrightarrow0\qquad(n\geq2).
\]
Consequently, if $\delta_{n+1}$ is surjective, then
\[
 \pi_nP\cong\pi_nX_1\times_{\pi_nX_{12}}\pi_nX_2.
\]
If $\delta_2$ is surjective, the analogous assertion in degree one is
an isomorphism of groups
\[
 \pi_1P\cong\pi_1X_1\times_{\pi_1X_{12}}\pi_1X_2
\]
for the chosen component.  Moreover, $P$ is connected if the double
coset set
\[
 f_{1*}(\pi_1X_1)\backslash\pi_1X_{12}/f_{2*}(\pi_1X_2)
\]
is a singleton.  These identifications are natural in pointed maps
of the three-corner diagrams.
\end{lemma}
\begin{proof}
The short exact sequences are obtained by taking the indicated
segments of the homotopy exact sequence in Corollary~\ref{dr}.
Surjectivity kills the preceding boundary map, and the projection
from $P$ then identifies its homotopy group with the displayed kernel
or group pullback.  The components of $P$ are the orbits of
$\pi_1X_1\times\pi_1X_2$ on $\pi_1X_{12}$ under left and right
multiplication, giving the double coset description.  Naturality
follows from the path-object model and its homotopy exact sequence.
\end{proof}

Corollary~\ref{dr} and Lemma~\ref{degreewise-pullback}, together with
the algebraic arithmetic square already proved in
Theorem~\ref{asn}, give the arithmetic square for nilpotent spaces of
finite type.
\begin{theorem}\label{space-arithmetic}
Let $X$ be a connected nilpotent Kan complex of finite type.  Set
$X_0=\mathbb Q_\infty X$, $X_p=(\mathbb F_p)_\infty X$, and
$B=\prod_pX_p$.  Then the natural map
\[
 X\longrightarrow\overline X
 :=\operatorname{holim}(X_0\longrightarrow\mathbb Q_\infty B
 \longleftarrow B)
\]
is a weak equivalence, hence a homotopy equivalence after realization.
\end{theorem}
\begin{proof}
For nilpotent finite-type $X$, its fundamental group is finitely
generated and each higher homotopy group is finitely generated abelian.
The nilpotent actions on higher homotopy groups have finite filtrations;
completion preserves these filtrations, with a bound independent of $p$
in each degree.  Thus $B$ is nilpotent as well.  Put $G_n=\pi_nX$.
Theorems~\ref{q} and~\ref{p} give
\[
 \pi_nX_0=(G_n)_0,\quad
 \pi_nB=\prod_p(G_n)_p^{\wedge},\quad
 \pi_n(\mathbb Q_\infty B)=\left(\prod_p(G_n)_p^{\wedge}\right)_0,
\]
with ordinary abelian rationalization when $n\geq2$.
For every $n\geq2$ the arithmetic square gives the short exact sequence
\[
 0\longrightarrow G_n\longrightarrow
 (G_n)_0\times\prod_p(G_n)_p^{\wedge}
 \xrightarrow{\delta_n}\left(\prod_p(G_n)_p^{\wedge}\right)_0
 \longrightarrow0.
\]
Here exactness, including surjectivity of $\delta_n$, is precisely
Theorem~\ref{asn} applied to the finitely generated abelian group
$G_n$.
Lemma~\ref{degreewise-pullback} therefore identifies
$\pi_n\overline X$ with $\ker\delta_n=G_n$.
For $n=1$, surjectivity of $\delta_2$ again gives injectivity into
$\pi_1X_0\times\pi_1B$; the image is the group-theoretic pullback,
which is $G_1$ by Theorem~\ref{asn}.
Finally, the factorization assertion of that theorem says that the two
images of $\pi_1X_0$ and $\pi_1B$ have a single double coset in
$\pi_1(\mathbb Q_\infty B)$.  Hence $\overline X$ is connected.
The natural map from $X$ induces precisely these identifications, and
is therefore a weak equivalence in all degrees.
All these isomorphisms are induced by the canonical comparison map;
no choices of arithmetic factorizations are used to define them.
They are therefore compatible with maps of spaces and, in particular,
with the transition maps of an inverse tower.
\end{proof}
\end{example}
\begin{example}[Homotopy inverse and direct limits]

The inverse category, $$I_s=\{... \rightarrow \bullet\rightarrow\bullet\rightarrow\bullet\}.$$ In this case $holim_{I_s} X$ is the so called homotopy inverse limit.
Dually, the direct category  $$D_s=\{... \leftarrow \bullet\leftarrow\bullet\leftarrow\bullet\}$$ here we call $hocolim_{D_s} X$ - the homotopy direct limit.

The categories $I_s$ and $D_s$ are particular examples of the so called categories of cofibrant and fibrant constants \cite[15.10]{Hir}. 
The corresponding categories  $SSets^{D_s}$ and  $SSets^{I_s}$ of $SSets$-diagrams have the canonical Reedy model categories structures, in which cofibrant diagrams $SSets^{D_s}$ are the direct diagrams, where each map is an embedding of simplicial sets and, dually, a fibrant diagram in $SSets^{I_s}$ is the inverse diagrams, whose initial object is Kan and whose transition maps are Kan fibrations. It follows from \cite[Theorem 19.9.1]{Hir} that for such (co)fibrant diagrams in $SSets^{I_s}$ ($SSets^{D_s}$) the natural maps $$lim \; X\rightarrow holim\; X\quad (hocolim\; X\rightarrow colim\; X)$$
are weak equivalences (and give homotopy equivalences after realization).
\end{example}
\begin{proposition}\label{hofubini}
For small categories $I,J$ and a diagram $X:I\times J\to SSets$, there
are natural weak equivalences
\[
 \operatorname{holim}_{I\times J}X\simeq
 \operatorname{holim}_I\operatorname{holim}_JX\simeq
 \operatorname{holim}_J\operatorname{holim}_IX,
\]
\[
 \operatorname{hocolim}_{I\times J}X\simeq
 \operatorname{hocolim}_I\operatorname{hocolim}_JX\simeq
 \operatorname{hocolim}_J\operatorname{hocolim}_IX.
\]
If $I$ is finite and acyclic and $J$ is filtered, the natural comparison
\[
 \operatorname{hocolim}_J\operatorname{holim}_IX
 \longrightarrow\operatorname{holim}_I\operatorname{hocolim}_JX
\]
is a weak equivalence, with fibrant replacements understood.  The last
assertion does not permit interchanging an infinite inverse tower with
a filtered homotopy colimit.
\end{proposition}
\begin{proof}
The first two assertions are the Fubini properties of homotopy
(co)limits; see \cite[Section~8.5.5]{MV}.
For the mixed assertion, a finite acyclic
category has only finitely many nondegenerate composable strings of
arrows.  Its homotopy limit is therefore a finite homotopy-limit
construction.  Filtered colimits of simplicial sets preserve weak
equivalences and commute with finite limits, and filtered colimits of
Kan complexes are Kan.  Applying these facts to a finite homotopy-limit
model proves the comparison; see also \cite[Theorem~14.17]{Hir2}.
\end{proof}

Homotopy invariance follows from the standard models.
For objectwise fibrant diagrams, the Bousfield--Kan end model for
$\operatorname{holim}_I X$ is the mapping space of diagrams from the
cofibrant resolution $N(I\downarrow -)$ of the constant point to $X$.
The simplicial model-category axiom implies that this mapping-space
functor sends weak equivalences between fibrant diagrams to weak
equivalences.  For homotopy colimits, use the simplicial replacement,
whose terms are coproducts of the entries $X(i)$ indexed by composable
strings of arrows.  Objectwise weak equivalences give degreewise weak
equivalences of these bisimplicial sets and hence weak equivalences of
their diagonals.  In the pointed setting the corresponding constructions
use a disjoint base point on the indexing nerves and pointed coproducts.

The category $SSets^I$ is also a closed simplicial model category (the so called Bousfield-Kan model structure \cite[\S 8, Ch. XI]{BK}) with weak equivalences the maps $\phi:\mathcal{X}\rightarrow \mathcal{Y}\in SSets^I$such that $\phi_i:\mathcal{X}\rightarrow \mathcal{Y}\in SSets$ is a weak equivalence for every $i\in I$. The fibrations in $SSets^I$ are maps $\phi:\mathcal{X}\rightarrow \mathcal{Y}\in SSets^I$ 
such that $\phi_i:\mathcal X(i)\rightarrow\mathcal Y(i)$ is a fibration
for every $i\in I$; cofibrations are characterized by the left lifting
property with respect to objectwise trivial fibrations.  The simplicial
structure is obtained by viewing objects of $SSets^I$ as simplicial objects
in $Sets^I$ \cite[Section~5]{Hir2}.

\begin{definition} Assume that $\mathcal{C}$ is a model category and $F:\mathcal{C}\rightarrow \mathcal{D}$ is a functor. Consider pairs $(G,s)$ consisting of a functor $G: Ho(\mathcal{C})\rightarrow \mathcal{D}$ and natural transformation $s:G\gamma\rightarrow F$. A left derived functor for $F$ is a pair $(LF,t)$ of this type which is universal from the left, in the sense that if $(G,s)$ is any such pair, then there exists a unique natural transformation $s':G\rightarrow LF$ such that the composite natural transformation $$G\gamma\xrightarrow{s'\cdot \gamma} (LF)\gamma\xrightarrow{t} F$$ 
is the natural transformation $s$.
\end{definition}

The following proposition of Quillen \cite{Qui1} gives a sufficient condition for the existence of left derived functors
\begin{proposition}
Let $\mathcal{C}$ is a model category  and $F:\mathcal{C}\rightarrow \mathcal{D}$ is a functor with the property that $F(f)$ is an isomorphism whenever $f$  is a weak equivalence between cofibrant objects 
in $\mathcal{C}$. Then the left derived functor $(LF,t)$ of $F$ exists, and for each cofibrant object $X$ of $\mathcal{C}$ the map
$$t_X:LF(X)\rightarrow F(X)$$
is an isomorphism.
\end{proposition}
We will also need Quillen's definition of a total derived functor.
\begin{definition}
Assume we are given $F:\mathcal{C}\rightarrow \mathcal{D}$ a functor between model categories. A total derived functor $LF$ for $F$ is a functor 
$$LF: Ho(\mathcal{C})\rightarrow Ho(\mathcal{D})$$
which is a left derived functor for the composite $\gamma_{\mathcal{D}}\cdot F: \mathcal{C}\rightarrow Ho(\mathcal{D})$.
Similarly, a total right derived functor $RF$ for $F$ is a functor $RF: Ho(\mathcal{C})\rightarrow Ho(\mathcal{D})$ which is a right derived functor for the composite $\gamma_{\mathcal{D}}\cdot F$.
\end{definition}

The constant-diagram (or diagonal) functor
\[
\Delta_I:SSets\longrightarrow SSets^I
\]
sends a simplicial set $K$ to the constant $I$-diagram with value $K$.
It fits into the familiar adjoint triple
\[
 \operatorname{colim}_I\dashv \Delta_I\dashv \operatorname{lim}_I.
\]

\begin{theorem}
\label{thm:derived-holim-hocolim}
Let $I$ be a small category.  The Bousfield--Kan homotopy colimit and
homotopy limit induce functors
\[
 \operatorname{hocolim}_I:Ho(SSets^I)\longrightarrow Ho(SSets),
 \qquad
 \operatorname{holim}_I:Ho(SSets^I)\longrightarrow Ho(SSets).
\]
With respect to the adjoint triple
$\operatorname{colim}_I\dashv\Delta_I\dashv\operatorname{lim}_I$, there
are natural isomorphisms
\[
 \operatorname{hocolim}_I\cong \mathbf L\operatorname{colim}_I,
 \qquad
 \operatorname{holim}_I\cong \mathbf R\operatorname{lim}_I.
\]
Thus the homotopy colimit is the total left derived functor of the
ordinary colimit, whereas the homotopy limit is the total right derived
functor of the ordinary limit.
\end{theorem}

\begin{proof}
Equip $SSets^I$ with a model structure whose weak equivalences are
objectwise weak equivalences.  If $QX\xrightarrow{\sim}X$ is a cofibrant
replacement of an $I$-diagram, then the Bousfield--Kan coend formula is
naturally weakly equivalent to $\operatorname{colim}_I QX$.  Consequently
it sends weak equivalences between cofibrant diagrams to weak
equivalences and satisfies the universal property of the total left
derived functor of $\operatorname{colim}_I$.

Dually, if $X\xrightarrow{\sim}RX$ is a fibrant replacement in an
appropriate diagram model structure, then the Bousfield--Kan end formula
is naturally weakly equivalent to $\operatorname{lim}_I RX$.  It therefore
satisfies the universal property of the total right derived functor of
$\operatorname{lim}_I$.  These identifications are natural in $X$; see
\cite[Theorems~19.9.1 and~19.9.2]{Hir}.
\end{proof}

Notice that the theorem does not assert that homotopy (co)limits are
derived functors of $\Delta_I$.  Rather, $\Delta_I$ is the middle functor
in the two adjunctions, while the derived functors are taken from its left
and right adjoints.  Equivalently, on homotopy categories one obtains the
derived adjunctions
\[
 \operatorname{hocolim}_I\dashv Ho(\Delta_I)
 \dashv \operatorname{holim}_I.
\]

\section{Prounipotent presentations and profinite generator spaces}\label{s5}
We use the generator-space and continuity conventions of
\cite[Sections~1--2]{MikhContinuous}.  For a set $S$, let
$S^+=S\sqcup\{*\}$ have isolated points $s\in S$ and cofinite
neighbourhoods of $*$.  A family indexed by $S$ is called convergent
to $1$ when it extends to a continuous pointed map from $S^+$.

In characteristic zero, put
\[
 E_S=k^S,\qquad E_S^\vee=\operatorname{Hom}_{\mathrm{cts}}(E_S,k)
       =k^{(S)}.
\]
Here $E_S$ has the product topology, $k$ is discrete, and $k^{(S)}$
is the direct sum, regarded as a discrete vector space.  Define
\[
 \widehat{\mathbb L}(E_S)
 =\varprojlim_{W,c}
   \mathbb L(E_S/W)/\gamma_{c+1}\mathbb L(E_S/W),
 \qquad F_k(S)=\exp\widehat{\mathbb L}(E_S),
\]
where $W$ runs through open subspaces of finite codimension, $c\geq1$,
and $\mathbb L$ is the ordinary free Lie algebra.  The displayed
quotients are finite-dimensional nilpotent Lie algebras.  On the
$k$-points of a prounipotent group we use the inverse-limit topology
with discrete finite-dimensional unipotent quotients.  The category
itself has prounipotent group schemes as objects and group-scheme
morphisms as maps; normal closures and quotients are taken in this
category.

The canonical generators of $F_k(S)$ are $\exp(e_s)$, where $e_s\in E_S$
is the coordinate vector.  Their images in every finite-dimensional
unipotent quotient are $1$ for all but finitely many $s$.  By the
universal property of the completed free Lie algebra, continuous
pointed assignments $S^+\to U(k)$ extend uniquely to morphisms
$F_k(S)\to U$; see \cite[Lemma~2.2]{MikhContinuous}.

In characteristic $p$ we work with pro-$p$ groups and continuous
homomorphisms.  Write $F_p(S)$ for the free pro-$p$ group on the
convergent family indexed by $S$, characterized by
\[
 \operatorname{Hom}_{\mathrm{cts}}(F_p(S),U)
 \cong\operatorname{Map}_{\mathrm{cts},*}(S^+,U).
\]
We write $F(S)$ for either construction, according to the category;
in formulas common to both cases, $U(k)$ denotes $U$ in the pro-$p$
case.  For infinite $S$, freeness here refers to the specified pointed
profinite generator space $S^+$.  The identifications with degreewise
completions of discrete presentation models used below concern finite
rank in each simplicial degree.

For finite $S$ and characteristic zero, $F_k(S)$ is the inverse limit
of the unipotent completions of $\Phi(S)/\gamma_s\Phi(S)$.
For arbitrary $S$, the coordinate Hopf algebra has the description
\[
 \mathcal O(F_k(S))\cong T(E_S^\vee)=T(k^{(S)}),
\]
with shuffle multiplication and deconcatenation coproduct
\cite[Example~4]{Mikh2023}.  Thus
\[
 \mathcal O(F_k(S))\cong
 \varinjlim_{T\subseteq S,\ |T|<\infty}\mathcal O(F_k(T)),
 \qquad
 F_k(S)\cong\varprojlim_{T\subseteq S,\ |T|<\infty}F_k(T).
\]
The transition maps retain the generators in $T$ and send the others
to $1$.  The corresponding pro-$p$ construction has coordinate algebra
given by locally constant $\mathbb F_p$-valued functions
\cite[Example~6]{Mikh2023}; it likewise satisfies
$F_p(S)\cong\varprojlim_{T\subseteq S,\ |T|<\infty}F_p(T)$.
For the generator spaces, $S^+\cong\varprojlim_T T^+$.
We also use the notation $(\widehat S^*,*)$ for this chosen pointed
profinite space, so that $F(\widehat S^*,*)$ means $F(S)$.
A decomposition into finitely many disjoint subsets gives the
corresponding coproduct in the chosen category; in particular,
$F(X\sqcup Y)\cong F(X)\amalg F(Y)$.

Lifts of an infinite family are always required to be convergent.
The existence of such lifts follows from the projective lifting
statement \cite[Proposition~2.3]{MikhContinuous}: if $q:U\twoheadrightarrow V$
is an epimorphism in the chosen category, every morphism $F(S)\to V$
lifts to a morphism $F(S)\to U$.  In characteristic zero, choose a
continuous linear section of the surjection
$\operatorname{Lie}(U)\to\operatorname{Lie}(V)$ and extend its composite
with the map on $E_S$ by the completed free Lie universal property.
The linear section need not preserve brackets.  In the pro-$p$ case
use projectivity of the free pro-$p$ group.  This produces one
continuous morphism and hence a convergent family of lifts in either
case, with no restriction on $|S|$.

\begin{definition}\label{d7}
Let $F=F(X)$ and let $R$ be a closed normal subgroup in the chosen
category.  Put
\[
 V_R=\begin{cases}
 (R/\overline{[R,F]})(k),&\operatorname{char}(k)=0,\\
 R/\overline{R^p[R,F]},&k=\mathbb F_p.
 \end{cases}
\]
A simplicial presentation with a specified relator basis consists of
a continuous pointed map $\tau:Y^+\to R(k)$ whose image normally
generates $R$, together with a topological linear isomorphism
$k^Y\xrightarrow{\sim}V_R$ sending $e_y$ to the class of
$r_y=\tau(y)$.  Equivalently, these classes are the coordinate
pseudobasis dual to a chosen basis of the discrete space
$V_R^\vee\cong k^{(Y)}$.  With $G=F/R$, the associated diagram is
\begin{equation}
\xymatrix{
 F(X\sqcup Y)\ar@<0ex>[r]^{d_0}\ar@<-2ex>[r]^{d_1}
 &F(X)\ar@<-2ex>[l]_{s_0}\ar[r]^{\pi}&G
}\label{5}
\end{equation}
where
\[
 d_0(x)=d_1(x)=x,\qquad d_0(y)=1,\qquad d_1(y)=r_y,
 \qquad s_0(x)=x.
\]
Continuity of $\tau$ and the universal property of $F(X\sqcup Y)$
make these maps morphisms in the chosen category.
\end{definition}

The basis condition makes $\tau$ injective; since $Y^+$ is compact
and $R(k)$ has the Hausdorff inverse-limit topology, $\tau$ identifies
$Y^+$ with its image as a pointed profinite space.
Continuous lifts of a prescribed coordinate pseudobasis of $V_R$
exist: apply the preceding projective lifting statement to the
quotient $R\to V_R$ and the morphism $F(Y)\to V_R$ defined by the
coordinate vectors.  In characteristic zero, $V_R$ here denotes the
corresponding commutative prounipotent group as well as its $k$-points.
The specified normal-generation and basis conditions are the data
required for the presentation theory of
\cite[Section~4.5]{Mikh2023}.  For the finite discrete subpresentations
used below, the next lemma verifies these conditions for the given
words of relations; continuity of their finite indexing family is
automatic.

\begin{lemma}\label{lem:relator-independence}
Let $(X_1\mid Y_1)$ be a finite subpresentation of a contractible
presentation $(X\mid Y)$.  The images of $r_y$, $y\in Y_1$, span a
free direct summand of $\mathbb Z^{(X_1)}$ with this specified basis.
Let $F=F_k(X_1)$ in characteristic zero, or let $F$ be the free pro-$p$
group on $X_1$ with $k=\mathbb F_p$, and let $R_1$ be the closed normal
subgroup generated by these relators.  There are canonical isomorphisms
\[
 k^{Y_1}\xrightarrow{\ \sim\ }V_1,\qquad
 V_1=\begin{cases}
 (R_1/\overline{[R_1,F]})(k),&\operatorname{char}(k)=0,\\
 R_1/\overline{R_1^p[R_1,F]},&k=\mathbb F_p,
 \end{cases}
\]
which send the coordinate vectors to the classes of the chosen relators.
Moreover, these relators extend to a free topological basis of $F$.
\end{lemma}
\begin{proof}
By Lemma~\ref{l.1.3}, the images of all $r_y$, $y\in Y$, form a basis
of $\mathbb Z^{(X)}$.  Let $L$ be the span of those indexed by $Y_1$.
The coordinate projection $\mathbb Z^{(X)}\to L$ restricts to a
retraction $\mathbb Z^{(X_1)}\to L$, since every $r_y$ with $y\in Y_1$
is a word in $X_1$.  This proves the direct-summand assertion, including
when the ambient contractible presentation is infinite.

Tensoring with $k$ gives an injection $k^{Y_1}\hookrightarrow k^{X_1}$.
Normal generation gives a surjection $k^{Y_1}\twoheadrightarrow V_1$:
conjugation by $F$ is trivial in $V_1$, so the images of the specified
normal generators span it.  In characteristic zero this is a statement
about vector groups; in characteristic $p$ it is a statement about
compact elementary abelian groups.  Since $Y_1$ is finite, its span is
closed in both cases.  The composite
\[
 k^{Y_1}\longrightarrow V_1\longrightarrow\operatorname{Ab}_k(F)
       \cong k^{X_1}
\]
is the preceding injection.  Thus the first map is also injective and
is the claimed canonical isomorphism.

Finally extend these vectors to a basis of $k^{X_1}$ and lift the added
vectors to $F$.  The basis theorem for finite-rank free prounipotent
groups, respectively the pro-$p$ Burnside basis theorem and the Hopfian
property of finitely generated profinite groups, shows that the resulting
family is a free topological basis.  This step takes place entirely
inside the completion of $\Phi(X_1)$; no basis theorem for an infinite
ambient pro-$p$ group is required.
\end{proof}

\begin{remark}
In particular, every finite contractible discrete presentation gives a
simplicial presentation with the specified relator basis of
Definition~\ref{d7}.  Indeed, the relators normally generate
$\Phi(X)$, so their images generate the whole closed normal subgroup
$R=F(X)$ after completion.  Lemma~\ref{l.1.3} gives
$|X|=|Y|=n$ and a unimodular exponent-sum matrix.  Over $k$ this
identifies
\[
 k^Y\xrightarrow{\sim}V_R=\operatorname{Ab}_k(F(X)),
 \qquad e_y\longmapsto\operatorname{ab}_k(r_y),
\]
where $\operatorname{Ab}_k(F)=F/\overline{F^p[F,F]}$ in the
pro-$p$ case.  The family is finite, hence convergent, and the basis
condition holds for the original words of relations.

This condition is the one imposed on the relators in
\cite[Section~4.5, Remark~9]{Mikh2023}.  In that same remark,
\emph{proper} means that inflation $H^1(G,k)\to H^1(F,k)$ is an
isomorphism.  For the full contractible presentation $G=1$, this map
is $0\to k^n$, so the presentation is non-proper whenever $n>0$.
More generally, for a finite subpresentation as in
Lemma~\ref{lem:relator-independence}, its source has dimension
$|X_1|-|Y_1|$ and its target has dimension $|X_1|$; thus it is
non-proper whenever $Y_1$ is nonempty.  The criteria in
\cite[Propositions~15 and~16]{Mikh2023} apply to these non-minimal
presentations with the specified relator basis.
\end{remark}

\begin{remark}\label{r7}
For a finite discrete presentation $(X\mid Y)$, form the rational or
pro-$p$ completion of its free simplicial presentation.  The images
of the given relators are a finite, hence continuous, family of normal
generators.  To apply the presentation theory of \cite{Mikh2023}, one
must additionally verify the specified basis condition of
Definition~\ref{d7}; it does not follow from normal generation alone.
For finite subpresentations of contractible presentations,
Lemma~\ref{lem:relator-independence} supplies the required passage from
the discrete cellular calculation to the completed relator-basis condition.
The free quotient identified below then permits application of the
cohomological asphericity criterion of
\cite[Propositions~15 and~16]{Mikh2023}.
\end{remark}
As in Section \ref{r1.2}, we consider a simplicial prounipotent presentation \eqref{5} as the second step of the Eilenberg-MacLane complex building and denote it by $F^{(1)}_{\bullet}$, we also define the second homotopy group of a prounipotent presentation as $\pi_1(F^{(1)}_{\bullet})$. 
For a connected CW-complex $T$, Kan constructs a free simplicial group
$B_T$ with one non-degenerate generator in degree $n$ for each
$(n+1)$-cell of $T$ and a weak equivalence
$B_T\simeq GS_1T$ \cite[Theorem~6.2]{Kan}.  Thus
\[
 \overline WB_T\simeq \overline WGS_1T\simeq S_1T.
\]
We fix the following model convention for the finite presentations below.
Write $K=K(X_1\mid Y_1)$ for the pointed finite presentation complex and
$P_\bullet=P_\bullet(X_1\mid Y_1)$ for its free simplicial presentation.
Each $P_n$ is a finitely generated free group, and there are no
non-degenerate free generators in degrees greater than one.  The reduced
singular complex $S_1K$ is used only for comparison:
\[
 P_\bullet\simeq G(S_1K),\qquad
 |\overline WP_\bullet|\simeq K.
\]
No degreewise finiteness assertion is made about $S_1K$ or its Kan loop
group.  Finiteness of the CW-complex $K$, finite generation of each $P_n$,
and finite homological type are distinct properties; here the first
ensures the latter two for the chosen presentation model.

For $R=\mathbb Q$ or $\mathbb F_p$, write
$\widehat P_\bullet=(P_\bullet)^{\wedge_R}$ for its dimensionwise
completion.  The dimensionwise completion theorem of Bousfield--Kan
\cite[Proposition~4.1]{BK}, applied to this free presentation model of
finite type, supplies the comparison
\[
 R_\infty K\simeq |\overline W\widehat P_\bullet|.
\]
This use of the completion-model theorem is separate from homotopy
invariance of $R_\infty$: we do not infer that dimensionwise completion
preserves arbitrary weak equivalences of simplicial groups.  Kernels,
Moore complexes and filtrations below are formed on the chosen
presentation model or its explicitly indicated completion.

We record the characteristic-zero facts needed for the finite-dimensional
quotients below.  Our convention for the Zariski closure of a subgroup
of a proalgebraic group is the inverse limit of its Zariski closures in
algebraic quotients, as in \cite[Section~2]{HM2003}.  The equivalence between
unipotent groups and finite-dimensional nilpotent Lie algebras, and its
extension to pro-objects, are recalled in
\cite[Section~3.1]{HainWeight2008}; see also \cite[Section~3]{Hain1993}
and \cite[Propositions~14.32, 14.36 and Theorem~14.37]{MilneAG2017}.

In particular, if $U$ is unipotent and $V\triangleleft U$ is a closed
$k$-subgroup, then $(U/V)(k)=U(k)/V(k)$; this is explicitly recalled in
\cite[Section~7.4, proof of Lemma~7.5]{HM2003} (arXiv version).
For prounipotent groups, the same surjectivity follows by applying
$\exp$ to the surjective quotient map of linearly compact Lie algebras.
Likewise, the image of a morphism of prounipotent groups is closed:
its Lie algebra is the image of a continuous linear map between
linearly compact spaces.  Such an image is closed, since continuous
duality identifies it with the annihilator of the kernel of the dual
map.  For the pro-vector-space conventions, see
\cite[Section~2.1]{HainWeight2008}.

\begin{lemma}\label{lem:unipotent-generation}
Let $U$ be a unipotent algebraic group over a field $k$ of characteristic
zero.  If closed $k$-subgroups $U_i$ generate $U$ as an algebraic group,
then
\[
 U(k)=\langle U_i(k):i\in I\rangle_{\mathrm{abs}}.
\]
\end{lemma}
\begin{proof}
Put $H=\langle U_i(k)\rangle_{\mathrm{abs}}$ and
$\mathfrak u=\operatorname{Lie}U$.  Write $C^r\mathfrak u$ for its lower
central series and $U_r=\exp(C^r\mathfrak u)$.
The images of $\operatorname{Lie}U_i$ span
$\mathfrak u/C^2\mathfrak u$: otherwise their inverse image in
$\mathfrak u$ would be a proper Lie subalgebra, whose exponential would
be a proper algebraic subgroup containing every $U_i$.
Under the exponential correspondence, the image of $U_i(k)$ in
$(U/U_2)(k)$ is precisely the corresponding linear subspace.
Consequently $H\to(U/U_2)(k)$ is surjective.

For every $r\geq2$, $C^r\mathfrak u/C^{r+1}\mathfrak u$ is spanned by
length-$r$ brackets from $\mathfrak u/C^2\mathfrak u$.
Choose lifts in $H$ of their entries.  The leading term of the logarithm
of the corresponding iterated group commutator is the Lie bracket.
Scalar coefficients can be absorbed in one entry, whose lift again
exists by surjectivity on the first quotient.  Products add these
leading terms.  It follows that
\[
 H\cap U_r(k)\longrightarrow (U_r/U_{r+1})(k)
\]
is surjective for every $r$.  Starting with an arbitrary element of
$U(k)$, successively correct it by elements of $H$ along this filtration.
The process ends because $\mathfrak u$ is nilpotent, proving $H=U(k)$.
\end{proof}

\begin{theorem}\label{main}
Let $(X_1\mid Y_1)$ be a finite simplicial subpresentation of a
contractible simplicial presentation $(X\mid Y)$.  Let $Z$ be a set of
cardinality $|X_1|-|Y_1|$.  For $R=\mathbb Q$ or $R=\mathbb F_p$,
\[
 R_\infty\overline W(X_1\mid Y_1)\simeq K(F_R(Z),1).
\]
Here $F_{\mathbb Q}(Z)$ denotes the rational points of the free
prounipotent group over $\mathbb Q$, and $F_{\mathbb F_p}(Z)$ denotes
the free pro-$p$ group.  In particular, the fundamental group is
isomorphic to $F_R(Z)$ and all higher homotopy groups vanish.
\end{theorem}

\begin{proof}
We first treat a field $k$ of characteristic zero.  By
Lemma~\ref{lem:relator-independence}, the specified relators extend to a
free topological basis $Y_1\sqcup Z$ of $F_k(X_1)$, and their images form
the canonical basis of the relative abelianization of their closed normal
closure.  Consequently,
\[
 F_k(X_1)/(Y_1)_{F_k(X_1)}\cong F_k(Z).
\]
The quotient has Hochschild cohomological dimension at most one (zero
if $Z$ is empty).  The cohomological asphericity criterion
\cite[Propositions~15 and~16]{Mikh2023}, in its version for non-minimal
presentations with the canonical relator basis, therefore gives
$\pi_2\overline W(X_1\mid Y_1)_k=0$.
For $k=\mathbb F_p$, the same lemma, applied inside the finite-rank free
pro-$p$ group on $X_1$, gives the complementary free basis and the same
conclusion in the mod-$p$ setting.

The criterion gives precisely
\[
 \pi_1(\widehat P_\bullet)
   =\pi_2(\overline W\widehat P_\bullet)=0.
\]
For higher homotopy groups we use the filtration argument below,
starting with the injectivity of the abelianized differential from
Lemma~\ref{lem:relator-independence}.
The dimensionwise completion model of Bousfield and Kan identifies the
classifying space of the completed simplicial group with the corresponding
completion of the original presentation complex.

Use the fixed model $P_\bullet=P_\bullet(X_1\mid Y_1)$ and its
dimensionwise completion $\widehat P_\bullet$, with
$R=\mathbb Q$ or $R=\mathbb F_p$.  Thus
$\widehat P_\bullet$ is a simplicial free prounipotent group in characteristic zero
and a simplicial free pro-$p$ group in characteristic $p$.  Its augmentation
has image
\[
 \pi_0\widehat P_\bullet\cong F_R(Z),
\]
where $Z$ is the complementary basis obtained above.  By construction,
$\widehat P_n$ is generated by degenerate elements for every $n\geq2$.

Put
\[
 N_2\widehat P=\ker d_0^2\cap\ker d_1^2,
 \qquad K_i=\ker d_i^1\quad(i=0,1).
\]
The prounipotent (respectively pro-$p$) Brown--Loday lemma, in the
presentation setting of \cite[Section~4.3]{Mikh2023}, gives
\begin{equation}\label{eq:completed-brown-loday}
 d_2^2(N_2\widehat P)=\overline{[K_0,K_1]}.
\end{equation}
Here the bar denotes topological closure in the pro-$p$ case and closed
subgroup-scheme closure in the prounipotent case.\footnote{For a topologically finitely generated profinite group $G$
and a closed normal subgroup $H$, Nikolov--Segal prove that the abstract
subgroup $[H,G]$ is closed; in particular, $[G,G]$ is closed.
See N.~Nikolov and D.~Segal, \emph{On normal subgroups of compact groups},
Theorem~1.1, \href{https://arxiv.org/abs/1310.3359}{arXiv:1310.3359}.
Finite generation of $H$ is not required.  This statement does not by
itself establish closedness of $[K_0,K_1]$ for two closed normal subgroups;
we retain the closure in \eqref{eq:completed-brown-loday}.}

We verify the completed identity using common quotients of the
$2$-truncated simplicial object $G=\widehat P_{\leq2}$.
For the simplicial profinite inverse-limit framework, see Quillen
\cite{QUILLEN69}.  Choose normal subgroups $U_n\triangleleft G_n$, $0\leq n\leq2$, with
finite quotients in the pro-$p$ case, or closed normal subgroup schemes
with finite-dimensional quotients in characteristic zero.  Set
\[
 V_n=\bigcap_{\substack{0\leq m\leq2\\
              \alpha:G_n\to G_m\ \mathrm{structural}}}
              \alpha^{-1}(U_m).
\]
There are only finitely many structural maps in this truncated simplex
category, including identities.  Thus the quotients $G_n/V_n$ are finite,
respectively finite-dimensional.  Closure under composition gives
$\beta(V_n)\subseteq V_m$ for every structural map
$\beta:G_n\to G_m$.  Thus they are quotients of the truncated simplicial object.
Since $V_n\subseteq U_n$, they form a cofinal system
$q_\lambda:G\twoheadrightarrow G_\lambda$ with
$G\cong\varprojlim_\lambda G_\lambda$.

The images of the degeneracies generate $G_{2,\lambda}$ abstractly:
this is immediate for finite groups and follows from
Lemma~\ref{lem:unipotent-generation} for unipotent groups.  Hence the discrete
Brown--Loday lemma \cite[Lemma~5.7]{BL} applies at every stage.
The kernels also lift: for $\bar x\in\ker d_i\subseteq G_{1,\lambda}$,
choose a lift $x\in G_1$ and replace it by
$x(s_0d_ix)^{-1}$.  The resulting lift belongs to $K_i$, so
$q_\lambda(K_i)=K_{i,\lambda}$.  Consequently, for $z\in N_2G$,
\[
 q_\lambda(d_2z)\in d_2N_2G_\lambda
 =[K_{0,\lambda},K_{1,\lambda}]
 =q_\lambda([K_0,K_1]).
\]
Cofinality now gives $d_2N_2G\subseteq\overline{[K_0,K_1]}$.
Conversely, for $x\in K_0$ and $y\in K_1$, the element
\[
 z(x,y)=[s_1x,s_1y(s_0y)^{-1}]
\]
satisfies $d_0z=d_1z=1$ and $d_2z=[x,y]$.
Thus $[K_0,K_1]\subseteq d_2N_2G$.  The latter image is closed by
compactness in the pro-$p$ case and by the closed-image property recalled before
Lemma~\ref{lem:unipotent-generation} in characteristic zero.  This proves \eqref{eq:completed-brown-loday}.

The identity identifies the kernel of the boundary map of the associated crossed
module with
\[
 \ker\!\left(K_0/\overline{[K_0,K_1]}\xrightarrow{\overline d_1}\widehat P_0\right)
 =\frac{K_0\cap K_1}{\overline{[K_0,K_1]}}
 =\pi_1\widehat P_\bullet=\pi_2(\overline W\widehat P_\bullet).
\]
The cohomological criterion applied above makes this kernel zero.

For the higher-degree argument, we next pass to the $R$-abelianization
\[
 A_\bullet=\operatorname{Ab}_R(\widehat P_\bullet),
\]
where $\operatorname{Ab}_{\mathbb Q}(\widehat P)=\widehat P/\overline{[\widehat P,\widehat P]}$ and
$\operatorname{Ab}_{\mathbb F_p}(\widehat P)=\widehat P/\overline{\widehat P^p[\widehat P,\widehat P]}$.  Since
$\widehat P_n$ is generated by degeneracies for $n\geq2$, the simplicial
$R$-module $A_\bullet$ satisfies
\[
 A_n=D_nA\qquad(n\geq2).
\]
Consequently its normalized complex is concentrated in degrees zero and
one:
\[
 N_nA=0\qquad(n\geq2).
\]
It remains to compute $\pi_1A$ using
Lemma~\ref{lem:relator-independence}.
The differential $N_1A\to N_0A$ is injective because the images of the
relators $Y_1$ span the direct summand established in
Lemma~\ref{lem:relator-independence}.  Explicitly, $N_1A=R^{Y_1}$,
$N_0A=R^{X_1}$, and this differential is the exponent-sum map
$e_y\mapsto\operatorname{ab}_R(r_y)$.  Hence, by the Dold--Kan correspondence,
\begin{equation}\label{eq:abelianization-acyclic}
 \pi_nA_\bullet=H_n(NA_\bullet)=0\qquad(n>0).
\end{equation}
Thus $A_\bullet$ is weakly equivalent to the constant simplicial
$R$-module $\operatorname{Ab}_R(F_R(Z))$.

Write $F_s\widehat P_\bullet$ for the dimensionwise closed lower central
series in characteristic zero.  In characteristic $p$, use instead the
Zassenhaus filtration
\[
 F_sG=D_sG=\overline{\prod_{ip^j\geq s}\gamma_i(G)^{p^j}},
 \qquad i\geq1,\ j\geq0.
\]
In both cases $F_1\widehat P=\widehat P$, and the filtration is complete
and separated.  Since every $\widehat P_n$ is free, the associated graded
simplicial object is the free graded Lie object on $A_\bullet$:
\[
 \operatorname{gr}_s\widehat P_\bullet
 \cong L_s^R(A_\bullet),\qquad s\geq1,
\]
where $L_s^{\mathbb Q}$ is the degree-$s$ part of the free Lie algebra and
$L_s^{\mathbb F_p}$ is its restricted analogue.  A weak equivalence of
simplicial $R$-modules is a simplicial homotopy equivalence, and the
dimensionwise functor $L_s^R$ preserves simplicial homotopies.  It follows
from \eqref{eq:abelianization-acyclic} that
\[
 \pi_n\operatorname{gr}_s\widehat P_\bullet
 =\pi_nL_s^R(A_\bullet)=0
 \qquad(n>0,\ s\geq1).
\]

For the simplicial pro-$p$ framework and its associated spectral
sequences, we refer to Quillen, \emph{An application of simplicial
profinite groups} \cite{QUILLEN69}; for the rational setting, see
\cite{Qui5}.  The chosen filtration gives the unstable Adams--Quillen
spectral sequence
\[
 E^1_{n,s}=\pi_n\operatorname{gr}_s\widehat P_\bullet.
\]
To pass to the completed group, consider the quotient tower.
Put $Q_s=\widehat P_\bullet/F_s\widehat P_\bullet$, so that $Q_1=1$.
The degreewise surjective maps $Q_{s+1}\to Q_s$ are fibrations of
underlying simplicial groups, with kernels
$F_s\widehat P_\bullet/F_{s+1}\widehat P_\bullet$.
In characteristic zero surjectivity on rational points follows from the
unipotent quotient property.  The long exact sequences of homotopy
groups and the vanishing of $\pi_n\operatorname{gr}_s\widehat P$ for
$n>0$ imply inductively
\[
 \pi_nQ_s=0\qquad(n>0,\ s\geq1).
\]
Completeness identifies $\widehat P_\bullet=\varprojlim_sQ_s$.
Since this is a tower of fibrations, its inverse limit computes its
homotopy inverse limit.  For $n\geq1$, the Milnor exact sequence has
outer terms
\[
 \varprojlim\nolimits_s^1\pi_{n+1}Q_s=0,
 \qquad \varprojlim_s\pi_nQ_s=0.
\]
(The degree-one statement is interpreted in the usual group/pointed-set
form.)  It follows that
\[
 \pi_n\widehat P_\bullet=0\qquad(n>0).
\]
Consequently
\[
 \widehat P_\bullet\simeq\operatorname{Const}(F_R(Z))
 \quad\text{and}\quad
 \overline W\widehat P_\bullet\simeq K(F_R(Z),1).
\]
The classifying space of the completed simplicial presentation, and hence
$R_\infty\overline W(X_1\mid Y_1)$, is therefore aspherical.
\end{proof}

\begin{proposition}\label{prop:keune-finite-type}
Let $k$ be a field of characteristic zero, and let
$X_\bullet\to G$ and $Y_\bullet\to H$ be free simplicial
prounipotent resolutions of finite type over $k$.  Here finite type means
that in every simplicial degree the free basis is finite, and the chosen
bases are compatible with degeneracies.  Every morphism $\alpha:G\to H$
extends to a simplicial morphism $X_\bullet\to Y_\bullet$.
Any two such extensions are simplicially homotopic over $\alpha$, with
all components of the homotopy morphisms of prounipotent groups.
In particular, two such resolutions of the same group are simplicially
homotopy equivalent over that group.
\end{proposition}
\begin{proof}
The proof of Keune's comparison theorem
\cite[Theorem~2]{Ke} applies without change in finite type.
Indeed, any assignment of a finite free basis to $k$-points of a
prounipotent group extends uniquely to a morphism of prounipotent groups.
The resolving property of the target provides the required lifts of
compatible boundaries.  In each degree only finitely many basis elements
must be lifted, while the values on degenerate basis elements are already
prescribed.  Thus Keune's inductive constructions of comparison maps and
simplicial homotopies take place in the prounipotent category itself.
Applying the comparison in both directions over the identity, and its
homotopy uniqueness to the composites, proves the final assertion.
For a full continuous lifting proof, including infinite bases compatible
with degeneracies, see \cite[Proposition~3.1 and Theorem~4.2]{MikhContinuous}.
\end{proof}

\begin{corollary}\label{cor:prounipotent-model-comparison}
In characteristic zero, the equivalence with the constant simplicial group
in Theorem~\ref{main} can be realized by morphisms of simplicial
prounipotent groups and simplicial homotopies whose components are
prounipotent morphisms.
\end{corollary}
\begin{proof}
The completed presentation $\widehat P_\bullet$ has a finite free basis
in every degree, compatible with degeneracies.  The preceding proof
establishes that its augmentation to $F_k(Z)$ is a resolution.
Since $F_k(Z)$ is free of finite rank, $\operatorname{Const}(F_k(Z))$
is also a free simplicial prounipotent resolution of finite type.
Proposition~\ref{prop:keune-finite-type} applies to these two resolutions.
\end{proof}

\section{Subcontractible presentations}\label{s6}

We apply the arithmetic square at each nilpotent stage and take the
homotopy inverse limit.
\begin{proposition}\label{cx5}
Let $T_s$ be an inverse tower of connected nilpotent finite-type spaces,
and write $D_s$ for the three-corner diagram
\[
 (T_s)_0\longrightarrow\mathbb Q_\infty\!\left(\prod_p(T_s)_p\right)
 \longleftarrow\prod_p(T_s)_p.
\]
Here $(T_s)_0=\mathbb Q_\infty T_s$ and
$(T_s)_p=(\mathbb F_p)_\infty T_s$.  If $T=\operatorname{holim}_sT_s$,
then there is a homotopy pullback square
\begin{equation}\label{as}
\xymatrix{
 T\ar[r]\ar[d] &\operatorname{holim}_s\prod_p(T_s)_p\ar[d]\\
 \operatorname{holim}_s(T_s)_0\ar[r] &
 \operatorname{holim}_s\mathbb Q_\infty(\prod_p(T_s)_p).}
\end{equation}
No identification of the bottom-right corner with
$\mathbb Q_\infty(\operatorname{holim}_s\prod_p(T_s)_p)$ is assumed.
\end{proposition}
\begin{proof}
Put $P_s=\operatorname{holim}_{I_{pb}}D_s$.
Theorem~\ref{space-arithmetic} and
Lemma~\ref{degreewise-pullback} give canonical isomorphisms
$\pi_nT_s\cong\pi_nP_s$ for all $n\geq1$, compatible with every
transition map.

After functorial fibrant replacement of the towers, Milnor's natural
exact sequences, for $n\geq2$, are
\[
\begin{split}
0\longrightarrow\varprojlim\nolimits_s^1\pi_{n+1}T_s
&\longrightarrow\pi_n\operatorname{holim}_sT_s
\longrightarrow\varprojlim_s\pi_nT_s\longrightarrow0,\\
0\longrightarrow\varprojlim\nolimits_s^1\pi_{n+1}P_s
&\longrightarrow\pi_n\operatorname{holim}_sP_s
\longrightarrow\varprojlim_s\pi_nP_s\longrightarrow0.
\end{split}
\]
The stagewise natural isomorphisms identify both outer terms and
hence the middle terms.  No vanishing of these derived limits is
needed for this comparison.  The corresponding degree-one sequence
is interpreted as an exact sequence of groups, and components by
the usual nonabelian $\varprojlim^1$ of the fundamental groups.
Equivalently, including all components and base points, homotopy
invariance of homotopy inverse limits applied to the objectwise
weak equivalence $T_s\to P_s$ gives
$\operatorname{holim}_sT_s\simeq\operatorname{holim}_sP_s$.

Finally, Proposition~\ref{hofubini} gives
\[
 T\simeq\operatorname{holim}_s\operatorname{holim}_{I_{pb}}D_s
 \simeq\operatorname{holim}_{I_{pb}}\operatorname{holim}_sD_s,
\]
which gives \eqref{as}.
\end{proof}

\begin{proposition}
\label{general-free-arithmetic}
Let $P_\bullet$ be an arbitrary free discrete simplicial group of
finite type: $P_n$ is free of finite rank for every $n$.  No nilpotency
assumption is imposed on $P_\bullet$ or on $\pi_0P_\bullet$.
Put $X=\overline WP_\bullet$.  All the following constructions are
performed degreewise.  Put
\[
 N_{s,\bullet}=P_\bullet/\gamma_sP_\bullet,\qquad
 \widehat P_{\mathbb Q,\bullet}=\varprojlim_s(N_{s,\bullet})_0,
 \qquad
 \widehat P_{p,\bullet}=\varprojlim_s(N_{s,\bullet})_p^{\wedge}.
\]
These are respectively the rational prounipotent and pro-$p$
completions of the original discrete free simplicial group.  Define
its adelic simplicial group by
\[
 \mathcal A(P)_\bullet=
 \varprojlim_s\varinjlim_{S\subset\pi,\ S\text{ finite}}
 \left(
 \prod_{p\in S}((N_{s,\bullet})_p^{\wedge})_0
 \times\prod_{p\notin S}(N_{s,\bullet})_p^{\wedge}
 \right).
\]
There are natural weak equivalences
\begin{equation}\label{general-group-reconstruction}
 \widehat P_{\mathbb Z,\bullet}:=
 \varprojlim_sN_{s,\bullet}\ \simeq\
 \operatorname{holim}\left(
 \widehat P_{\mathbb Q,\bullet}\longrightarrow\mathcal A(P)_\bullet
 \longleftarrow\prod_p\widehat P_{p,\bullet}\right)
\end{equation}
in simplicial groups, and consequently
\begin{equation}\label{general-space-reconstruction}
 \mathbb Z_\infty X\simeq
 \operatorname{holim}\left(
 \overline W\widehat P_{\mathbb Q,\bullet}
 \longrightarrow\overline W\mathcal A(P)_\bullet
 \longleftarrow\overline W\!\left(\prod_p\widehat P_{p,\bullet}\right)
 \right).
\end{equation}
The homotopy limit in these formulas is over the three-object cospan.
The order $\varprojlim_s\varinjlim_S$ in the adelic entry is retained.
No asphericity assumption on $X$ is required.  If $X\to\mathbb Z_\infty X$
is a weak equivalence, the left side of
\eqref{general-space-reconstruction} can be replaced by $X$.
\end{proposition}
\begin{proof}
In each degree
$n$, the group $N_{s,n}=P_n/\gamma_sP_n$ is finitely generated free
nilpotent.  Proposition~\ref{asnfn} supplies its arithmetic square,
and the factorization in Theorem~\ref{asn} makes the corresponding
square of group nerves a homotopy pullback.  In the two-sided bar model,
the rational and pro-$p$ product groups act on the adelic group by
left and right multiplication.  The factorization makes this action
transitive in every degree, and the stabilizer of the identity is
$N_{s,n}$.  The inclusion of this stabilizer into the action groupoid
is therefore an equivalence in every degree, naturally in the
simplicial operators.  Its diagonal is a weak equivalence.  The
two-sided bar construction models the homotopy pullback of the three
classifying spaces.  This proves the required simplicial homotopy
pullback assertion; applying the derived loop functor gives the
corresponding assertion for simplicial groups.

Now take the homotopy inverse limit in $s$ and commute the two homotopy
limits by Proposition~\ref{hofubini}.  Every tower involved has
surjective maps in each simplicial degree and hence consists of
fibrations of simplicial groups.  For the adelic tower, lift each
component while keeping the same finite set $S$; this proves its
surjectivity as well.  Thus these homotopy inverse limits are computed
by ordinary inverse limits.  The filtered colimit defining the adelic
entry is also computed degreewise; on group nerves its injective
transition maps are cofibrations.  Finally, ordinary inverse limits
commute with products.  These observations prove
\eqref{general-group-reconstruction}.  Apply the derived classifying-space
functor and the lower-central tower model for integral completion
\cite[Chapter~IV]{BK} to obtain \eqref{general-space-reconstruction}.
\end{proof}

\subsection{The adelic limit and rational pronilpotent completion}
\label{adelic-explicit}
We give a concrete description of the adelic entry and distinguish it
from an interchange of an inverse limit and a filtered colimit.  In this
subsection we index the nilpotent quotients by their class: for a free
group $F$ of finite rank put
\[
 G_s=F/\gamma_{s+1}F,\quad N_{s,p}=(G_s)_p^{\wedge},\quad
 U_{s,p}=(N_{s,p})_0,\quad N_s=\prod_pN_{s,p}.
\]
This shifts the index by one relative to
Proposition~\ref{general-free-arithmetic}.  By Proposition~\ref{asnfn}
we may regard $N_{s,p}$ as a subgroup of $U_{s,p}$ and set
\[
 A_{s,S}=\prod_{p\in S}U_{s,p}\times\prod_{p\notin S}N_{s,p},
 \qquad A_s=\bigcup_{S\text{ finite}}A_{s,S}=(N_s)_0.
\]
Let
\[
 N_p=F_p^{\wedge}=\varprojlim_sN_{s,p},\qquad
 U_p=\varprojlim_sU_{s,p},\qquad N=\prod_pN_p.
\]
All embeddings and projections here are the natural ones.

\begin{proposition}
\label{adelic-completion-description}
There is a natural group isomorphism
\begin{equation}\label{adelic-coordinate-description}
 \mathcal A(F):=\varprojlim_s A_s
 \cong\left\{(u_p)_p\in\prod_pU_p:\ 
 \substack{\text{for each }s,\ u_p\bmod s\in N_{s,p}\\
 \text{for all but finitely many }p}\right\}.
\end{equation}
Here $u_p\bmod s$ denotes the projection to $U_{s,p}$.
Let
\[
 \mathcal B(F)=\bigcup_{S\text{ finite}}
 \left(\prod_{p\in S}U_p\times\prod_{p\notin S}N_p\right),
 \qquad K_s=\ker(\mathcal B(F)\longrightarrow A_s).
\]
Then $\mathcal B(F)\to A_s$ is surjective for every $s$, and
\begin{equation}\label{adelic-restricted-completion}
 \mathcal B(F)/K_s\cong A_s,\qquad
 \mathcal A(F)\cong\varprojlim_s\mathcal B(F)/K_s.
\end{equation}
Thus $\mathcal A(F)$ is the separated completion of the restricted
product $\mathcal B(F)$ for this filtration.
Moreover, if $\gamma_jN$ denotes the \emph{abstract} lower central series,
then
\begin{equation}\label{adelic-rational-completion}
 \mathcal A(F)\cong\widehat N_{\mathbb Q}^{\mathrm{nil}}
 :=\varprojlim_s(N/\gamma_{s+1}N)_0.
\end{equation}
These identifications are natural in homomorphisms between finite-rank
free groups, and hence apply degreewise to the free simplicial groups
in Proposition~\ref{general-free-arithmetic}.
\end{proposition}
\begin{proof}
A compatible family $a_s\in A_s$ determines, for each prime $p$, a
compatible family of $p$-coordinates and hence $u_p\in U_p$.
Membership in $A_s$ is exactly the finiteness condition in
\eqref{adelic-coordinate-description}.  Conversely, that condition
allows a tuple $(u_p)_p$ to be projected to $A_s$ on every level.
These two constructions are inverse homomorphisms.

For fixed $S$, inverse limits commute with products, giving
\[
 \varprojlim_sA_{s,S}
 =\prod_{p\in S}U_p\times\prod_{p\notin S}N_p.
\]
To lift $a_s\in A_{s,S}$ to $\mathcal B(F)$, lift its finitely many
coordinates in $U_{s,p}$ to $U_p$, and all remaining coordinates in
$N_{s,p}$ to $N_p$.  These projections are surjective, as their
transition towers consist of surjections.  The lifts retain the same
finite set $S$.  The first isomorphism of
\eqref{adelic-restricted-completion} follows, and then so does the
second.  The intersection of the $K_s$ is trivial by the coordinate
description, so this is a separated completion.

To prove \eqref{adelic-rational-completion}, first observe that $N$ is
a topologically finitely generated profinite group.  If $x_1,\ldots,x_r$
is a basis of $F$, their diagonal images topologically generate $N$:
every finite continuous quotient is a product of finite $p$-groups for
finitely many distinct primes, and the images of these generators
surject onto every factor.  A subdirect subgroup of a product of finite
groups of pairwise coprime orders is the whole product.

By the theorem of Nikolov--Segal, every term of the abstract lower
central series of a finitely generated profinite group is closed
\cite{NS2007}.  In particular, all $\gamma_jN$ and $\gamma_jN_p$ are
closed.  Coordinate projections give one inclusion in
\[
 \gamma_jN=\prod_p\gamma_jN_p.
\]
For the reverse inclusion, $\gamma_jN$ contains $\gamma_jN_p$ supported
in each individual coordinate, and therefore contains all finite-support
products of these subgroups.  They are dense in the displayed product;
closedness gives the reverse inclusion.

For a free pro-$p$ group, quotienting by its closed lower central term
is the pro-$p$ completion of the corresponding free nilpotent group.
This also follows directly from their common universal property for
maps to finite $p$-groups of class at most $s$.  Consequently,
\[
 N/\gamma_{s+1}N
 \cong\prod_p(N_p/\gamma_{s+1}N_p)
 \cong\prod_pN_{s,p}=N_s.
\]
Proposition~\ref{asnfn} now identifies the rationalization of this
quotient with $A_s$.  Taking inverse limits proves
\eqref{adelic-rational-completion}.  All maps used are induced by
projections, inclusions and rationalization, which also proves naturality.
\end{proof}

\begin{proposition}
\label{adelic-noninterchange}
For $F=F(x,y)$, the canonical map
\[
 \iota:\varinjlim_S\varprojlim_sA_{s,S}
       \longrightarrow\varprojlim_s\varinjlim_SA_{s,S}
\]
is injective but not surjective.  Under the preceding identifications,
it is the inclusion $\mathcal B(F)\hookrightarrow\mathcal A(F)$.
\end{proposition}
\begin{proof}
The map sends the class of a compatible family $(a_s)_s$ represented
at a fixed finite $S$ to the same family in $\varprojlim_sA_s$.
Two representatives can be compared at the union of their finite sets
of primes.  Since $A_{s,S}\to A_s$ is injective for every $s$, equality
of their images implies equality of the representatives.  Thus $\iota$
is injective.

Choose distinct primes $p_j$ for $j\geq2$, and for each $j$ choose a
basic Hall commutator $c_j\in F$ of weight $j$.  Define $a_s$ by
\[
 (a_s)_{p_j}=(c_j\bmod\gamma_{s+1}F)^{1/p_j}\in U_{s,p_j},
\]
and put all other prime coordinates equal to $1$.  Roots exist uniquely
in these rational nilpotent groups.  When $j>s$, this coordinate is $1$,
so $a_s\in A_{s,\{p_2,\ldots,p_s\}}$.  Projections preserve unique
roots; hence $(a_s)_s$ is a compatible family in $\mathcal A(F)$.

At level $s=j$, however, $c_j^{1/p_j}$ does not lie in $N_{j,p_j}$.
Indeed, the central layer of weight $j$ in $N_{j,p_j}$ is the free
$\mathbb Z_{p_j}$-module on the basic commutators of weight $j$.
The element in question has coordinate $1/p_j$ at $c_j$.  Its projection
to class $j-1$ is trivial by uniqueness of roots, so if it lay in
$N_{j,p_j}$ it would belong to that central integral layer, which is
impossible.  Every fixed set $S$ representing $(a_s)_s$ would therefore
have to contain all $p_j$.  No finite $S$ does, proving non-surjectivity.
The constant simplicial group on $F(x,y)$ gives the same example within
free simplicial groups of finite rank in every degree.
\end{proof}

\begin{remark}\label{adelic-completion-comparison}
Proposition~\ref{adelic-completion-description} identifies the rational
pronilpotent completion of $N$ with the adelic limit:
$\widehat N_{\mathbb Q}^{\mathrm{nil}}\cong\mathcal A(F)$.

On every nilpotent stage there is also the standard natural
identification
\[
 \mathbb Q_\infty BN_s\simeq B(N_s)_0=BA_s.
\]
Since the maps $A_{s+1}\to A_s$ are surjective, passage to the tower
and the classifying-space construction gives
\[
 \operatorname{holim}_s\mathbb Q_\infty BN_s
 \simeq B\!\left(\varprojlim_sA_s\right)=B\mathcal A(F).
\]
The exceptional
finite set of primes may depend on the nilpotent stage, as expressed
by \eqref{adelic-coordinate-description};
Proposition~\ref{adelic-noninterchange} shows that replacing it by a
single finite set changes the limit.

The arithmetic square uses this homotopy inverse limit; no
identification of $B\mathcal A(F)$ with $\mathbb Q_\infty BN$ is required.
\end{remark}

\begin{theorem}\label{c4}
Let $(X_1\mid Y_1)$ be a finite subpresentation of a contractible
presentation $(X\mid Y)$, and let $|Z|=|X_1|-|Y_1|$.  For
$R\in\{\mathbb Q,\mathbb F_p,\mathbb Z\}$ one has
\[
 R_\infty S_1K(X_1\mid Y_1)\simeq K(F_R(Z),1),
\]
where the field cases use the notation of Theorem~\ref{main}, and
\[
 F_{\mathbb Z}(Z)=\widehat{\Phi(Z)}_{\mathrm{nil}}
 =\varprojlim_{s\geq2}\Phi(Z)/\gamma_s\Phi(Z).
\]
\end{theorem}
\begin{proof}
The field assertions, including the identification of the fundamental
groups, are Theorem~\ref{main}.  For the integral assertion, let
$P_\bullet$ be the free discrete simplicial presentation of
$K(X_1\mid Y_1)$.  Its generators in degrees above one are degeneracies.
Lemma~\ref{l.1.3} shows that the abelianized words of relations form a
primitive independent family in $\mathbb Z^{X_1}$.  Complete this family
to an integral basis and choose words $z\in\Phi(X_1)$ lifting the added
basis vectors, indexed by $Z$.

Fix $s\geq2$.  The substitution sending formal generators $u_y$ to the
words of relations $r_y$, and the complementary generators to the
chosen words $z$, is an isomorphism on abelianizations.  On each
lower-central quotient of a free group it induces the corresponding
isomorphism of free Lie powers.  Induction through the finite central
series therefore shows that it is an isomorphism on the free nilpotent
quotient of class at most $s-1$.

Consequently $P_\bullet/\gamma_sP_\bullet$ can be written as the free
nilpotent simplicial group of class at most $s-1$ on a pointed simplicial
set with constant vertices $Z$ and, for each $y$, an edge from $u_y$ to
the base point.  Its edge satisfies $d_0y=1$, $d_1y=u_y$.
Contract these edges to the base point while fixing $Z$.  Applying the
free nilpotent group functor degreewise gives
\[
 P_\bullet/\gamma_sP_\bullet\simeq
 \operatorname{Const}G_s,\qquad G_s=\Phi(Z)/\gamma_s\Phi(Z).
\]
These equivalences and contractions are compatible with $s$, since the
same words and the same edge contractions were chosen at every stage.
In particular,
$T_s:=\overline W(P_\bullet/\gamma_sP_\bullet)\simeq K(G_s,1)$.

The lower-central tower model for integral Bousfield--Kan completion
\cite[Chapter~IV]{BK} now gives
\[
 \mathbb Z_\infty S_1K(X_1\mid Y_1)
 \simeq\operatorname{holim}_sT_s
 \simeq\operatorname{holim}_sK(G_s,1).
\]
The quotient maps $G_{s+1}\to G_s$ are surjective.  Their nerves form a
tower of Kan fibrations, so its ordinary limit computes its homotopy
limit.  Since the nerve of a group commutes with limits, the last space
is $K(\varprojlim_sG_s,1)$, as required.

We can also describe this tower arithmetically, retaining the order of
limits.  Apply Proposition~\ref{asnfn} to $G_s$ and write
\[
 A_{s,S}=\prod_{p\in S}U_{s,p}(\mathbb Q_p)
          \times\prod_{p\notin S}(G_s)_p^{\wedge},\qquad
 A_s=\varinjlim_S A_{s,S}.
\]
Here $U_{s,p}$ is the unipotent $\mathbb Q_p$-group associated with
$G_s$.  This notation refers to nilpotent groups at a fixed stage;
there is no scalar extension from $\mathbb F_p$ to $\mathbb Q_p$.
Because $G_s$ is a finitely generated torsion-free nilpotent group,
$(G_s)_p^{\wedge}$ embeds in $U_{s,p}(\mathbb Q_p)$.  Thus all transition
maps $A_{s,S}\to A_{s,S'}$, $S\subseteq S'$, are injective, and their
nerves are monomorphisms, hence cofibrations of simplicial sets.
One may use a cofinal sequence of finite sets of primes.  The ordinary
colimit then computes the homotopy colimit.  Since the nerve commutes
with filtered colimits,
\[
 \operatorname{hocolim}_S BA_{s,S}\simeq
 \varinjlim_S BA_{s,S}=B\!\left(\varinjlim_S A_{s,S}\right)=BA_s.
\]
This cofibration assertion is made in simplicial sets, not in simplicial
groups.

Likewise, the transition maps of each of the towers
\[
 G_s,\qquad (G_s)_0,\qquad \prod_p(G_s)_p^{\wedge}
\]
are surjective.  Their nerves are therefore towers of Kan fibrations,
so their homotopy inverse limits are computed by ordinary inverse
limits.  Define the concrete groups
\[
 F_{\mathbb Z}(Z)=\varprojlim_sG_s,\qquad
 F_{\mathbb Q}(Z)=\varprojlim_s(G_s)_0,\qquad
 F_{\mathbb F_p}(Z)=\varprojlim_s(G_s)_p^{\wedge}.
\]
The latter two agree with the free prounipotent and free pro-$p$ groups
used above.  For the pro-$p$ identification, every finite $p$-group
quotient of $\Phi(Z)$ is nilpotent and hence factors through some $G_s$.
The rational identification is the lower-central construction of free
prounipotent completion.  Ordinary limits commute with products, giving
\[
 \varprojlim_s\prod_p(G_s)_p^{\wedge}
 =\prod_pF_{\mathbb F_p}(Z).
\]

The adelic tower also has surjective transition maps.  Indeed, an
 element of $A_s$ belongs to $A_{s,S}$ for some finite $S$.  Each of its
 components lifts at stage $s+1$: the rational unipotent maps and the
 pro-$p$ maps are surjective.  These lifts belong to $A_{s+1,S}$, with
 the same set $S$.  Thus $A_{s+1}\to A_s$ is surjective.  Consequently,
\[
 \operatorname{holim}_s BA_s\simeq B\mathcal A(Z),\qquad
 \mathcal A(Z):=\varprojlim_s\varinjlim_S A_{s,S}.
\]
The order of the two ordinary limits in this definition is essential;
no interchange is asserted.

For this particular presentation, the passage to the limit may also
be read directly from Milnor's sequence.  The compatible contractions
above give $T_s\simeq K(G_s,1)$ at every stage, so, for $n\geq2$,
\[
 0\longrightarrow\varprojlim\nolimits_s^1\pi_{n+1}T_s
 \longrightarrow\pi_n(\operatorname{holim}_sT_s)
 \longrightarrow\varprojlim_s\pi_nT_s\longrightarrow0
\]
has both outer terms zero.  Moreover,
$\pi_1(\operatorname{holim}_sT_s)\cong\varprojlim_sG_s$;
the limit is connected because the tower of $G_s$ is surjective.
The same observations apply to the towers of group nerves in the
other entries.  The vanishing follows from the stagewise models and surjectivity.

The factorization assertion in Theorem~\ref{asn} makes the homotopy
pullback at each nilpotent stage connected, with fundamental group
$G_s$ and no higher homotopy.  Proposition~\ref{cx5} therefore gives the
following \emph{noncommutative arithmetic square}:
\begin{equation}\label{as2}
\xymatrix{
 \operatorname{holim}_s T_s\ar[r]\ar[d] &
 B\!\left(\prod_pF_{\mathbb F_p}(Z)\right)\ar[d]\\
 BF_{\mathbb Q}(Z)\ar[r] & B\mathcal A(Z).}
\end{equation}
Here $B$ denotes the nerve of an abstract group.  The upper-left corner
is the homotopy pullback of the other three corners, which are now
classifying spaces of explicit groups (equivalently, the
$\overline W$-constructions on constant simplicial groups).
The maps are induced by the compatible maps at the nilpotent stages.
The upper-left corner is also identified with $BF_{\mathbb Z}(Z)$.
Finally, $\overline W$ is a right Quillen functor and hence its derived
functor preserves homotopy limits.  
\end{proof}

\begin{remark}
The contractions in Theorem~\ref{c4} use the subcontractibility
hypothesis, whereas Proposition~\ref{general-free-arithmetic} applies
to every free discrete simplicial group of finite type.
\end{remark}

\bibliography{Mikhovich_Whitehead_Arxiv1}{}
\bibliographystyle{plain}

\end{document}